\pdfoutput=1
\documentclass[11pt]{article}

\usepackage[margin=1in, tmargin = 1in]{geometry}

\usepackage[T1]{fontenc}

\usepackage{listings}
\usepackage{csquotes}
\usepackage{amssymb,amsmath,amsthm}
\usepackage{verbatim,gensymb,titlesec,url}
\titlespacing\section{0pt}{12pt plus 4pt minus 2pt}{4pt plus 2pt minus 0pt}
\titlespacing\subsection{0pt}{12pt plus 4pt minus 2pt}{2pt plus 2pt minus 0pt}
\titlespacing\subsubsection{0pt}{12pt plus 4pt minus 2pt}{2pt plus 2pt minus 0pt}

\usepackage{array,float}
\usepackage{multirow, pbox}
\usepackage{adjustbox,color}

\usepackage{bm}
\usepackage{caption}
\usepackage{subcaption}
\usepackage{placeins}

\newtheorem{theorem}{Theorem}[section]

\newtheorem{remark}{Remark}[section]

\newtheorem{proposition}{Proposition}

\usepackage{lastpage, pbox}
\usepackage{wrapfig,enumerate,mathtools}
\usepackage{graphicx}
\usepackage[update]{epstopdf}
\usepackage{soul}

\usepackage{enumerate}
\usepackage{enumitem}
\usepackage{hyperref}

\definecolor{wine}{RGB}{165,0,33}
\definecolor{greelue}{RGB}{0,180,180}
\title{\bf A hybrid CNN–adjoint optimization framework for reconstruction of viscoelastic tissue properties in magnetic resonance elastography}

\author{Anwesa Dey\thanks{Department of Mathematics, The University of Utah, address USA. e-mail: anwesary@math.utah.edu} \and Johann Rudi \thanks{Department of Mathematics, Virginia Tech, address USA. e-mail: jrudi@vt.edu} \and Elena Cherkaev \thanks{Department of Mathematics, The University of Utah, address,
address, USA. e-mail: elena@math.utah.edu}}

\date{}

\begin{document}

\maketitle

\begin{abstract}

Magnetic resonance elastography (MRE) is a noninvasive imaging modality for quantifying the viscoelastic properties of soft tissues by observing shear wave propagation. Recovering the complex-valued shear modulus from measured displacement fields leads to a severely ill-posed inverse problem, particularly in the presence of noise and a limited number of boundary excitations. In this work, we investigate adjoint-based optimization, convolutional neural network (CNN) reconstruction, and a hybrid framework that combines the two approaches. The forward model is based on a scalar form of the modified stationary Stokes system with a complex shear modulus. We first present a rigorous theoretical analysis of an adjoint-based optimization formulation, establishing well-posedness of the forward boundary-value problem, existence of minimizers, and first-order optimality conditions, and implement a nonlinear conjugate-gradient method with Armijo line search. The optimization method can refine coefficient reconstructions through PDE-constrained data fitting and regularization, but its performance can depend strongly on the choice of the initial guess. We then construct a two-dimensional CNN that learns a nonlinear mapping from complex-valued internal displacement measurements to spatially varying complex-valued shear modulus fields, showing substantially improved initialization. The CNN provides rapid, spatially informative reconstructions once trained, but its accuracy depends on the availability of sufficiently large and representative training datasets, the training process can be computationally expensive, and its predictions are not explicitly constrained by the governing PDE.
Motivated by the complementary strengths of physics-based and data-driven methods, we propose a hybrid approach in which a CNN provides an informed initial guess for the adjoint-based optimization, yielding faster convergence and improved accuracy. The CNN is trained on coefficient fields containing individual perturbations and is evaluated both on individual structures and on more challenging coefficient fields containing both types of perturbations within the same domain, although such combined configurations are not included in the training data. Numerical experiments demonstrate that the CNN provides accurate initial reconstructions and generalizes to previously unseen combined configurations, while the subsequent PDE-constrained optimization further refines the reconstructed coefficient. The results indicate the potential of combining data-driven initialization with physics-based optimization for accurate and efficient reconstruction of complex-valued shear modulus fields in MRE.

\end{abstract}

{\em Keywords}: {Optimization, 
Non-linear conjugate gradient descent, machine learning, tomography, inverse problems.}

\section{Introduction} \label{sec-intro}
Magnetic resonance elastography (MRE) is a non-invasive imaging modality that enables the quantitative assessment of tissue mechanical properties by visualizing the propagation of shear waves through soft tissues. It has become an essential tool for diagnosing and monitoring diseases characterized by alterations in tissue stiffness, such as liver fibrosis, tumors, and neurodegenerative disorders. Unlike conventional magento resonance imaging (MRI), which primarily provides structural or functional information, MRE adds a mechanical dimension to imaging, offering insights into the viscoelastic behavior of tissues in vivo.

The MRE process consists of three main stages. First, low-frequency mechanical vibrations are introduced into the tissue using an external actuator, inducing shear waves that traverse the imaging domain. Second, specialized MRI pulse sequences, typically phase-contrast imaging synchronized with the vibration, are employed to capture the resulting displacement field in three orthogonal directions. This yields a high-resolution spatiotemporal representation of wave propagation. Finally, the acquired displacement data are used to solve an inverse problem governed by physical laws, typically modeled by the time-harmonic or stationary form of the Stokes or Navier--Stokes equations. This inversion yields a spatial map of the complex shear modulus, which encodes both tissue stiffness and damping characteristics. However, despite its clinical potential, the inverse problem in MRE is notoriously ill-posed due to measurement noise, model uncertainties, and the often partial or corrupted nature of the displacement data. 

There are several works that consider models for shear wave displacements in MRE. In \cite{Higashimori2007,Nakamura2008}, the authors use scalar and simplified version of the time-dependent Kelvin--Voigt model for the shear wave displacements in MRE. In \cite{Hubmer2018}, a steady state model for linearized elasticity with
displacement-traction boundary conditions is proposed. The authors in \cite{Jiang2011} use different models like the general time-dependent Kelvin--Voigt viscoelastic model in
an isotropic, nearly incompressible medium, a steady state version of the same 
model, known as the stationary viscoelastic equation, and the modified stationary Stokes model for MRE, which is
motivated by dominance of transverse waves.

Correspondingly, the inverse problem of obtaining the shear modulus has also been studied both from the theoretical and computational aspects in some works. In \cite{VanHouten1999}, a finite element-based nonlinear inversion scheme is proposed for MRE based on small overlapping subzones of the total region of interest. Some uniqueness results for an inverse problem to reconstruct the
viscoelastic coefficients from a single interior data, given by a time-periodic
solution of the time-dependent simplified Kelvin--Voigt model was given by the authors in \cite{Higashimori2007}. The authors in \cite{Nakamura2008} proposed a reconstruction procedure for obtaining the viscoelastic coefficients of a Kelvin--Vogit model, assuming the coefficients are known on the boundary. A stability analysis for reconstruction of viscoelastic coefficients
of the Stokes equation in MRE was presented in \cite{Ammari2015}. The authors in \cite{Hubmer2018} proposed a Landweber scheme to reconstruct the Lam\'e parameters from a static displacement field. Recently, in \cite{JiangNakamura2019,Jiang2021}, a Levenberg-Marquardt scheme and its analysis for numerically solving the MRE inverse problem was studied.

Broadly, coefficient reconstruction problems arising in MRE belong to the class of PDE-constrained inverse problems, for which optimization methods provide a general computational framework. Adjoint-based methods are particularly attractive for problems involving spatially distributed parameters, since they allow derivatives of the objective functional with respect to the unknown coefficient to be computed efficiently. Such approaches have been studied extensively for large-scale PDE-constrained inverse problems. For example, Epanomeritakis et al.~\cite{epanomeritakis2008newton} developed a Newton--CG method for elastic full-waveform inversion. Related adjoint-based approaches have also been used for the recovery of rheological parameters in flow models. Petra et al.~\cite{petra2012inexact} considered inversion of basal sliding and rheological parameters in a nonlinear Stokes ice-sheet model, while Worthen et al.~\cite{worthen2014adjoint} developed an adjoint-based inversion framework for rheological parameters in nonlinear viscous mantle flow.

Alongside optimization-based approaches, learning-based methods have become increasingly important in computational inverse problems. Neural networks can provide rapid approximations of inverse maps from observed data to unknown parameters, while more general operator-learning approaches seek to learn mappings between function spaces associated with PDE-governed systems \cite{kovachki2023neural}. Learning-based methods have also been investigated in wave-based inverse problems. Sun and Demanet~\cite{sun_demanet_fwi} used deep learning for low-frequency extrapolation in full-waveform inversion, while Sun et al.~\cite{sun_real_data} developed a learning-based strategy applicable to field data without requiring corresponding real-data labels. Li, Demanet, and Zepeda-N'u~nez~\cite{li_wideband} proposed a multi-frequency neural-network approach for stable and efficient inversion of wave data.

These developments have motivated increasing interest in combining data-driven methods with physics-based computational models. Ghattas and Willcox~\cite{ghattas2021learning} discuss this connection from the perspective of inverse problems and model reduction, emphasizing the complementary roles of physical models and observational data. Data-driven approximations have also been investigated within computational frameworks for inverse problems \cite{bigoni2020data}. In this setting, learning-based methods need not replace physics-based inversion; instead, learned reconstructions can provide useful information for initializing or accelerating subsequent model-based optimization. This perspective motivates the hybrid reconstruction strategy considered in this work.

We contribute to the field of MRE inverse problems by carrying out a theoretical analysis of an optimization-based framework for the reconstruction of a complex-valued shear modulus, including the well-posedness of the forward problem, existence of minimizers, and derivation of the corresponding first-order optimality conditions. We then investigate CNN-based and hybrid reconstruction frameworks in two dimensions. In the hybrid approach, the CNN reconstruction is used as an informed initial guess for the adjoint-based optimization. The CNN is trained on shear modulus fields containing individual Gaussian and mollified elliptical perturbations, and its performance is first evaluated on coefficient fields containing individual inclusions. We then consider coefficient fields containing both types of perturbations simultaneously, although such combined configurations are not included in the training data, in order to investigate the generalization of the learned reconstruction. Finally, we study the performance of the CNN and hybrid frameworks in the presence of measurement noise.

The paper is organized as follows. In the next section, we describe the stationary Stokes models used for the MRE inverse problem. In Section~3, we formulate the adjoint-based optimization problem for the reconstruction of the shear modulus and present the corresponding theoretical analysis. In Section~4, we introduce the CNN-based and hybrid reconstruction frameworks and present two-dimensional numerical experiments for noiseless and noisy displacement data. The final section summarizes the main conclusions and discusses directions for future work.

\section{Modified stationary Stokes system} 
\label{sec-MSS}
For the forthcoming dicussion, we define the following spaces and norms: The space $C^{0,1}(\bar{\Omega})$ denotes the space of Lipschitz continuous
functions on $\bar{\Omega}$, while $H^s(\Omega)$ denotes the
$L^2$-based Sobolev space of order $s$. We also use
$W^{k,\infty}(\Omega)$ to denote the Sobolev space of functions whose
weak derivatives up to order $k$ are essentially bounded. Thus,
\[
\|v\|_{W^{0,\infty}(\Omega)}
=
\|v\|_{L^\infty(\Omega)},
\]
and
\[
\|v\|_{W^{2,\infty}(\Omega)}
=
\max_{|\alpha|\leq 2}
\|\partial^\alpha v\|_{L^\infty(\Omega)}.
\]
For a collection of solutions
$\mathbf{u}=(u^1,\ldots,u^I)$, we define
\[
\|\mathbf{u}\|_{W^{2,\infty}(\Omega)}
=
\max_{1\leq j\leq I}
\|u^j\|_{W^{2,\infty}(\Omega)}.
\]

An appropriate partial differential equation (PDE) model for modeling the processes in magnetic resonance elastography is given by the following modified stationary Stokes system \cite{Jiang2021}:
\begin{equation}\label{eq:mss}
     \begin{aligned} 
      \nabla\cdot [A(\nabla u + \nabla u^T)]- \nabla p + \rho \omega^2 u = 0,& \quad\text{in}\hspace{2mm} \Omega \subset \mathbb{R}^n,\\
      \nabla \cdot u = 0,&\quad\text{in}\hspace{2mm} \Omega,\\
      u = f,&\quad \text{on}\hspace{2mm} \partial\Omega,\\
    \end{aligned}
\end{equation}
where $\Omega$, an open subset of $\mathbb{R}^n$, is the region of interest, \textit{u} is the wave displacement vector, $p$ is the pressure, $A$ is the scalar viscoelastic shear modulus, $\omega>0$ is the angular frequency. The corresponding inverse problem in MRE can then be stated as follows:  Let $(u, p) \in H^{1}(\Omega; \mathbb{C}^d) \times L^{2}(\Omega; \mathbb{C})/\mathbb{C}$ with a known
$f \in H^{1/2}(\partial\Omega; \mathbb{C}^d)$ be the solution of \eqref{eq:mss}. Using the MRE data $u\in \Omega$, find the unknown viscoelastic shear modulus $A$.

\subsection{Scalar model for the modified stationary Stokes system}

Under the assumption on incompressibility $\nabla \cdot u = 0$, the modified stationary Stokes system  reduces to the following scalar equation \cite{Jiang2021}:
\begin{equation}\label{eq:cplxPDE}
    \begin{aligned}
      \nabla\cdot [A(x)\nabla u] + \rho \omega^2 u = 0, &\quad \text{in}\hspace{2mm} \Omega,\\
      u = f, &\quad \text{on}\hspace{2mm} \partial\Omega.\\
    \end{aligned}
\end{equation} 
For both models, the viscoelastic shear modulus $A(x)$ is modeled as the sum of a constant baseline value and a perturbation. The constant baseline value represents the average shear modulus of the liver tissue in the absence of any abnormalities while the perturbation function represents the deviations from this baseline due to abnormalities in the liver tissue.

\section{Reconstruction of shear modulus function using adjoint-based optimization}
\label{sec-classicalopt}
Liver tissues are soft tissues that are viscoelastic in nature, exhibiting both elastic and viscous properties. In such a material, the viscoelastic shear modulus function \(A(x)\) is modeled as the sum of a constant baseline value $(1+i)$ and a complex-valued perturbation, $B(x) = B_1(x) + i B_2(x)$, where $B_1$ is such that $Re(A(x)) = 1+B_1(x) > 0$. The real part $1+B_1(x)$, which is the elastic component of the shear modulus, is called the storage modulus. It quantifies the energy stored in the liver tissues during deformation caused by shear stress, that is recoverable when the deforming force is removed. The imaginary part $1+B_2(x)$, which is the viscous component of the shear modulus, is called the loss modulus. It quantifies the energy loss during wave propagation through the liver tissues.

In this section, we aim to solve the inverse problem corresponding to equation \eqref{eq:cplxPDE} by reconstructing the complex-valued shear modulus function perturbation $B(x)\in L^2(\Omega)$, with $\Omega \subset \mathbb{R}^n$, given the displacement data $u(x)$, using an adjoint-based optimization method.  For this purpose, we first note that the use of a complex shear modulus function \(A(x) = A_1(x) +iA_2(x)\) in the complex PDE state equation \eqref{eq:cplxPDE} results in a complex-valued wave displacement, \(u = u_1+iu_2\). We then aim to obtain $B_1(x), B_2(x) \in L^2(\Omega)$ by minimizing the following functional:
For a collection of prescribed Dirichlet boundary data
$\{f^j\}_{j=1}^{I}$, let $u^j$ denote the corresponding solution of
\eqref{eq:cplxPDE}, and let $u^{\delta,j}$ denote the associated
measured displacement field. We consider the following PDE-constrained
optimization problem:
\begin{equation}\label{eq:cplxcostfn}
\begin{aligned}
\min_{B_1,B_2\in L^2(\Omega)}
J(\{u^j\}_{j=1}^{I},B_1,B_2)
:={}&
\frac{\alpha}{2}
\sum_{j=1}^{I}
\int_\Omega
\left|u^j(x)-u^{\delta,j}(x)\right|^2\,dx
\\
&+
\frac{\beta_1}{2}
\int_\Omega |B_1(x)|^2\,dx
+
\frac{\beta_2}{2}
\int_\Omega |B_2(x)|^2\,dx ,
\end{aligned}
\end{equation}
subject to
\begin{equation}\label{eq:cplxPDE_multi}
\begin{aligned}
\nabla\cdot\left[A(x)\nabla u^j\right]
+\rho\omega^2 u^j &=0,
&&\text{in } \Omega,\\
u^j&=f^j,
&&\text{on } \partial\Omega,
\qquad j=1,\ldots,I,
\end{aligned}
\end{equation}
where
\[
A(x)=(1+i)+B(x),
\qquad
B(x)=B_1(x)+iB_2(x).
\]
In the functional $J$, the first term represents the least-squares data fitting term to determine the fit of the measured MRE data $u^\delta$. The norm $\|\cdot \|$ stands for the complex modulus. The remaining two terms are the standard Tikhonov regularization for the unknown functions $B_1,B_2$.

\subsection{Theoretical results}

We now present some theoretical results of the MRE inverse problem and the associated optimization problem \eqref{eq:cplxcostfn}--\eqref{eq:cplxPDE}. We start with the well-posedness of the solutions of \eqref{eq:cplxPDE}.
\begin{theorem}[Th. 7.1 \cite{Jiang2011}, Th. 8.12 \cite{Gilbarg1977}]\label{th:MREpdeexistence}
For a given $f\in H^{1/2}(\partial\Omega)$  and $B\in L^\infty(\Omega)$, there exists an unique solution of \eqref{eq:cplxPDE} $u\in H^1(\Omega).$ Furthermore, the following stability estimate is satisfied
\[
\|u\|_{H^1(\Omega)} \leq C_1 \|f\|_{H^{1/2}(\partial\Omega)},
\]
where $C_1$ is a constant dependent on $\|B\|_2$. Furthermore, if $B$ is uniformly Lipschitz continuous and $\partial\Omega\in C^2$, then we have 
$u\in H^2(\Omega)$ and 
\[
\|u\|_{H^2(\Omega)} \leq C_2 \|f\|_{H^{1/2}(\partial\Omega)},
\]
where $C_2$ is a constant dependent on $\|B\|_2$.
\end{theorem}
We next state the solvability of the MRE inverse problem as given in the seminal paper \cite[Th. 2.1]{Bal2013}. 
\begin{theorem}\label{th:inverse}
Let $B(x)\in C^{0,1}(\bar{\Omega})\cap H^{n/2+4+\epsilon}(\bar{\Omega}),$ with $\Omega \subset \mathbb{R}^n$. Then for $I = I_n = \dfrac{1}{2}n(n+3)$, corresponding to the $I$ boundary conditions $(f^j)_{1\leq j \leq I}$, the solutions $(u^j)_{1\leq j \leq I}$ of \eqref{eq:cplxPDE} uniquely determine $B(x),$ where $\Omega \subset \mathbb{R}^n$. Furthermore, we have the following estimate
\[
\|B\|_{W^{0,\infty}(\Omega)} \leq D \|(u^1,\cdots,u^I)\|_{W^{2,\infty}(\Omega)}.
\]
\end{theorem}

\begin{remark}
The aforementioned Theorem \ref{th:inverse} implies that in one dimension, two pairs of boundary and interior datum guarantees unique recovery of $B$, while in two dimensions, five pairs of boundary and interior datum guarantees unique recovery of $B$. However, these are not optimal estimates as demonstrated in \cite[Section 4]{Bal2012}. Specifically, one isotropic coefficient $B(x)$ can be reconstructed from one measurement in any dimensions. 
\end{remark}

We next focus on the theoretical properties of the optimization problem \eqref{eq:cplxcostfn}. For this purpose, we define the following admissible set for $B$: 
\[B(x) \in U_{ad} = \left\{ B=B_1+iB_2\in L^2(\Omega;\mathbb{C}) : \|B\|_{L^2(\Omega)}\leq D,~ 1 + B_1(x) \geq a_0 >0 \right\}.
\]
\begin{proposition}\label{proposition3}
The objective functional $J$ given in \eqref{eq:cplxcostfn}
is sequentially weakly lower semicontinuous (w.l.s.c.) on $L^2(\Omega)$;
that is, if
$
B_k\rightharpoonup B
\quad\text{weakly in }L^2(\Omega),
$
then
\[
J(B)\leq\liminf_{k\to\infty}J(B_k).
\]
Moreover, $J$ is coercive on $L^2(\Omega)$, in the sense that
\[
\|B_k\|_{L^2(\Omega)}\to\infty
\quad\Longrightarrow\quad
J(B_k)\to\infty,
\]
and $J$ is Fr\'echet differentiable.
\end{proposition}

\begin{proof}
The coercivity property of $J$ follows from the presence of the terms $\|B_i\|_2$. For proving the w.l.s.c property, we first note that $J$ are not convex in $B$. However, $U_{ad}$ is convex and closed and also weakly sequentially compact.
This is because if $\{B^n\}\in U_{ad}$, then $\|B^n\|_{L^2}\leq D$ for all $n\in\mathbb{N}$. Thus, the sequence is uniformly bounded. Since $L^2(\Omega)$ is a Hilbert space, there is a  subsequence $\{B^{n_k}\}$ such that $B^{n_k}\rightharpoonup B\in L^2(\Omega)$. Also, we have that $\|B\|_{L^2}\leq \lim\inf\|B^n\|_{L^2}\leq D$. Thus, $B\in U_{ad}$. Next, we consider the set 
\[
U_{\eta} = \lbrace B \in U_{ad}: J(B) \leq \eta \rbrace.
\]
Since $J$ is continuous, the set $U_{\eta}$ is closed for all $\eta \in \mathbb{R}$. Thus, $U_{\eta}$ is a closed subset of a weakly sequentially compact space $U_{ad}$ and is weakly sequentially closed for all $\eta\in\mathbb{R}$ \cite{Beer1992}.
This implies that $J$ is weakly sequentially lower semi-continuous. Finally, the Fr\'echet differentiability of $J$ follows from similar arguments as given in \cite{Annunziato2013}. 
\end{proof}
We now state and prove the existence of a minimizer of $J$, given in \eqref{eq:cplxcostfn}.
\begin{theorem}\label{th:existence}
Let
$
f^j\in H^{1/2}(\partial\Omega),~ j=1,\ldots,I.
$
Denote $u^j(B)$ as the unique solution of \eqref{eq:cplxPDE_multi}, for every $B\in U_{ad}$ and each prescribed Dirichlet
boundary datum $f^j$, which satisfies the stability estimate
of Theorem~\ref{th:MREpdeexistence}. Let $J$ be the multiple-data objective functional defined in
\eqref{eq:cplxcostfn}. Then there exists
$B^*\in U_{ad}$
and an associated collection of states
$\mathbf{u}^*=(u^{1,*},\ldots,u^{I,*})$
such that each $u^{j,*}=u^j(B^*)$ solves
\eqref{eq:cplxPDE_multi} with boundary datum $f^j$, and
\[
J(\mathbf{u}^*,B^*)
=
\min_{B\in U_{ad}}J(\mathbf{u}(B),B),
\]
where
\[
\mathbf{u}(B)
=
(u^1(B),\ldots,u^I(B)).
\]
\end{theorem}

\begin{proof}
Since all terms in the objective functional are nonnegative,
\[
J(\mathbf{u}(B),B)\geq 0,
\]
and hence the infimum of $J$ over $U_{ad}$ is finite. Let
$\{B^m\}\subset U_{ad}$ be a minimizing sequence; that is,
\[
\lim_{m\to\infty}
J(\mathbf{u}(B^m),B^m)
=
\inf_{B\in U_{ad}}
J(\mathbf{u}(B),B).
\]
By the definition of $U_{ad}$,
\[
\|B^m\|_{L^2(\Omega)}\leq D
\]
for every $m$, and therefore $\{B^m\}$ is bounded in
$L^2(\Omega;\mathbb{C})$. Since $L^2(\Omega;\mathbb{C})$ is a
Hilbert space, there exists a subsequence, again denoted by
$\{B^m\}$, and some $B^*\in L^2(\Omega;\mathbb{C})$ such that
\[
B^m\rightharpoonup B^*
\qquad\text{weakly in }L^2(\Omega;\mathbb{C}).
\]
Moreover, $U_{ad}$ is closed and convex and is therefore weakly
closed. Consequently,
\[
B^*\in U_{ad}.
\]
For each $j=1,\ldots,I$, let
\[
u^{j,m}=u^j(B^m)
\]
denote the solution of \eqref{eq:cplxPDE_multi} corresponding to
$B^m$ and the fixed boundary datum $f^j$. By the stability estimate
of Theorem~\ref{th:MREpdeexistence}, the sequence
$\{u^{j,m}\}_m$ is bounded in $H^2(\Omega)$. Since $I$ is finite, a
diagonal subsequence may be selected such that, for every
$j=1,\ldots,I$,
\[
u^{j,m}\rightharpoonup u^{j,*}
\qquad\text{weakly in }H^2(\Omega).
\]
By the Rellich--Kondrachov compactness theorem,
\[
H^2(\Omega)\Subset H^1(\Omega),
\]
under the regularity assumptions on $\Omega$, and hence
\[
u^{j,m}\to u^{j,*}
\qquad\text{strongly in }H^1(\Omega).
\]
In particular,
\[
u^{j,m}\to u^{j,*}
\quad\text{and}\quad
\nabla u^{j,m}\to\nabla u^{j,*}
\qquad\text{strongly in }L^2(\Omega).
\]
We next show that $u^{j,*}$ satisfies the state equation corresponding
to $B^*$. For any $v\in C_c^\infty(\Omega)$, the weak formulation
satisfied by $u^{j,m}$ is
\[
\int_\Omega
\big((1+i)+B^m\big)
\nabla u^{j,m}\cdot\nabla v\,dx
-
\int_\Omega
\rho\omega^2u^{j,m}v\,dx
=0.
\]
For the second term, strong convergence in $L^2(\Omega)$ gives
\[
\int_\Omega
\rho\omega^2u^{j,m}v\,dx
\longrightarrow
\int_\Omega
\rho\omega^2u^{j,*}v\,dx.
\]
For the coefficient-dependent term, write
\[
\begin{aligned}
&
\int_\Omega
B^m\nabla u^{j,m}\cdot\nabla v\,dx
-
\int_\Omega
B^*\nabla u^{j,*}\cdot\nabla v\,dx
\\
&\quad =
\int_\Omega
B^m
\big(\nabla u^{j,m}-\nabla u^{j,*}\big)
\cdot\nabla v\,dx
+
\int_\Omega
(B^m-B^*)
\nabla u^{j,*}\cdot\nabla v\,dx.
\end{aligned}
\]
Since $\{B^m\}$ is bounded in $L^2(\Omega)$,
$\nabla u^{j,m}\to\nabla u^{j,*}$ strongly in $L^2(\Omega)$, and
$\nabla v\in L^\infty(\Omega)$, the first term on the right tends to
zero. Furthermore,
\[
\nabla u^{j,*}\cdot\nabla v\in L^2(\Omega),
\]
so the weak convergence
$B^m\rightharpoonup B^*$ in $L^2(\Omega)$ implies that the second
term also tends to zero. Consequently,
\[
\int_\Omega
B^m\nabla u^{j,m}\cdot\nabla v\,dx
\longrightarrow
\int_\Omega
B^*\nabla u^{j,*}\cdot\nabla v\,dx.
\]
Together with the strong convergence of $\nabla u^{j,m}$, this yields
\[
\int_\Omega
\big((1+i)+B^*\big)
\nabla u^{j,*}\cdot\nabla v\,dx
-
\int_\Omega
\rho\omega^2u^{j,*}v\,dx
=0
\]
for every $v\in C_c^\infty(\Omega)$.

Now, the trace operator is continuous from $H^1(\Omega)$ to
$H^{1/2}(\partial\Omega)$. Since
\[
u^{j,m}=f^j
\qquad\text{on }\partial\Omega
\]
for every $m$ and
$u^{j,m}\to u^{j,*}$ in $H^1(\Omega)$, it follows that
\[
u^{j,*}=f^j
\qquad\text{on }\partial\Omega.
\]
Hence $u^{j,*}$ is the solution of the state equation corresponding
to $B^*$ and the boundary datum $f^j$, for each $j=1,\ldots,I$.

Finally, by the sequential weak lower semicontinuity of the objective
functional,
\[
\begin{aligned}
J(\mathbf{u}^*,B^*)
&\leq
\liminf_{m\to\infty}
J(\mathbf{u}(B^m),B^m)
\\
&=
\inf_{B\in U_{ad}}
J(\mathbf{u}(B),B).
\end{aligned}
\]
Since $B^*\in U_{ad}$, the reverse inequality follows from the
definition of the infimum. Therefore,
\[
J(\mathbf{u}^*,B^*)
=
\min_{B\in U_{ad}}
J(\mathbf{u}(B),B),
\]
and $(\mathbf{u}^*,B^*)$ is a minimization pair.
\end{proof}

\subsection{Optimality conditions}
We now characterize the solutions of the minimization problem \eqref{eq:cplxcostfn}-\eqref{eq:cplxPDE_multi}, with $I=2$, using the first order necessary conditions for optimality. For this purpose, we note that one can rewrite \eqref{eq:cplxcostfn}-\eqref{eq:cplxPDE_multi} as follows:
\begin{equation}\label{eq:costfn}
\begin{aligned}
\min_{B_1, B_2 \in L^2(\Omega)} J(u_1, u_2, B_1, B_2):= &\frac{\alpha}{2} \int_\Omega \left| u_1(x) - u_1^\delta(x) \right|^2 + \left| u_2(x) - u_2^\delta(x) \right|^2 \, dx\\ 
 +& \frac{\beta_1}{2} \int_\Omega \left| B_1(x) \right|^2 \, dx 
+ \frac{\beta_2}{2} \int_\Omega \left| B_2(x) \right|^2 \, dx
\end{aligned}
\end{equation}
subject to the system of PDEs:
\begin{equation}\label{eq:PDE}
\begin{aligned}
\nabla\cdot [(1+B_1(x))\nabla u_1-(1+B_2(x))\nabla u_2] + \rho \omega^2 u_1 = 0, &\quad \text{in}\hspace{2mm} \Omega,\\
\nabla\cdot [(1+B_1(x))\nabla u_2+(1+B_2(x))\nabla u_1] + \rho \omega^2 u_2 = 0, &\quad \text{in}\hspace{2mm} \Omega,\\
u_1 = f_1, &\quad \text{on}\hspace{2mm} \partial\Omega,\\
u_2 = f_2, &\quad \text{on}\hspace{2mm} \partial\Omega.
\end{aligned}
\end{equation} 
The system of PDEs is obtained by expanding the complex valued PDE \eqref{eq:cplxPDE_multi} and separating out the real and the imaginary parts. This formulation is independent of any complex numbers and we now show the equivalence of the optimality system obtained by considering the complex form and system of PDEs form of the state equations, given in \eqref{eq:cplxPDE_multi} and \eqref{eq:PDE}.
\begin{theorem}

The first order optimality system corresponding to the solution of \eqref{eq:cplxcostfn}-\eqref{eq:cplxPDE_multi} is the same as the first order optimality system corresponding to the solution of \eqref{eq:costfn}-\eqref{eq:PDE}.
\end{theorem}
\begin{proof}
Since the cost functional \( J(u, B_1, B_2) \) is real-valued, and optimization must be performed over real-valued functionals, we define the Lagrangian as  \begin{equation}\label{eq: cplxlagr}
\mathcal{L} = J(u, B_1, B_2) + \operatorname{Re}\left\{ \langle F(u, A), \bar{q} \rangle_C \right\} 
\end{equation}   
where \(F(u,B) = F_1(u,B) + iF_2(u,B)\) denotes the solution of the complex-valued PDE \eqref{eq:cplxPDE_multi}, \( q = q_1 + iq_2 \) is the corresponding complex-valued adjoint variable, and $\langle \cdot,\cdot \rangle_C$ is the complex inner product in $L^2(\Omega)$. Simplifying \eqref{eq: cplxlagr}, we obtain:
\begin{equation}\label{eq:cplxlagr2}
\mathcal{L} = J(u, B_1, B_2) +  \langle F_1(u, B), q_1 \rangle + \langle F_2(u, B), q_2 \rangle,
\end{equation}
where $\langle \cdot,\cdot \rangle$ is the standard inner product in $L^2(\Omega)$.
Using the Gatueaux derivative, the adjoint partial differential equations are given as:
\begin{equation}\label{eq:cplxadjPDE}
\left\{
\begin{aligned}
\nabla \cdot \left( (1+B_1)\nabla q_1 - (1+B_2) \nabla q_2 \right) + \rho\omega^2 q_1 &= -\alpha(u_1 - u_1^\delta), \\[1.5ex]
\nabla \cdot \left( (1+B_1)\nabla q_2 + (1+B_2) \nabla q_1 \right) + \rho\omega^2 q_2 &= -\alpha(u_2 - u_2^\delta),
\end{aligned}
\right.
\end{equation}
for all \(q_1,\, q_2 \in H_f^1(\Omega)\).
In a similar way, the optimality conditions are the given as:
\begin{equation}\label{eq:cplxopti}
\left\{
\begin{aligned}
\dfrac{\partial \mathcal{L}}{\partial B_1} &= \beta_1B_1 - \nabla u_1\cdot\nabla q_1 - \nabla u_2\cdot\nabla q_2, \\[1.5ex]
\dfrac{\partial \mathcal{L}}{\partial B_2}&= \beta_2B_2 + \nabla u_2\cdot\nabla q_1 - \nabla u_1\cdot\nabla q_2,
\end{aligned}
\right.
\end{equation}
On the other hand, if we separate \(F(u,B)\) into its real and imaginary parts and combine them in a system of PDEs, then the Lagrangian we obtain \(\mathcal{L}\) is given by:
\begin{equation}\label{eq:systlagr}
\mathcal{L} = J(u, B_1, B_2) +  \langle F_1(u, B), q_1 \rangle + \langle F_2(u, B), q_2 \rangle.
\end{equation}
As the Lagrangian obtained in \eqref{eq:systlagr} is identical to equation \eqref{eq:cplxlagr2}, therefore the adjoint equations \(\dfrac{\partial \mathcal{L}}{\partial u_1}\) and \(\dfrac{\partial \mathcal{L}}{\partial u_2}\) obtained are same as \eqref{eq: cplxlagr}. Similarly, the optimality conditions in this case are same as \eqref{eq:cplxopti}.
\end{proof}

\subsection{Numerical method}
To implement the optimality system \eqref{eq:PDE}, \eqref{eq:cplxadjPDE}, and \eqref{eq:cplxopti}, we implement a non-linear conjugate gradient (NCG) method with Armijo line search, an extension of the classical conjugate gradient algorithm designed for nonlinear optimization problems. A key advantage of NCG over Newton-type schemes is its computational efficiency. Although Newton’s method can converge more rapidly, forming and inverting the Hessian is expensive and often prohibitive, especially in high-dimensional problems or when the sample size \(n\) is large. In contrast, NCG requires only gradient evaluations, which leads to significantly faster overall convergence in these settings.

We begin the algorithm with an initial guess \(\mathbf{B}_0\) for the perturbation function vector. The first search direction is taken to be the negative gradient of $J$ with respect to $B$
\[
\textbf{d}_0 = \textbf{g}_0 
= -\nabla_{\mathbf{B}} J(\mathbf{B}_0).
\]
For subsequent iterations (\(k = 0,1,2,\ldots\)), the search directions are defined recursively by
\[
\mathbf{d}_k = -\nabla J(\mathbf{B}_k) + {\gamma}_{k-1} \mathbf{d}_{k-1},
\]
where ${\gamma}_{k-1}$ is a scalar computed using the formula of Dai–Yuan and is given as follows:
\[
\gamma_i = \frac{\langle \textbf{g}_k, \textbf{g}_k \rangle}{\langle \mathbf{d}_{k-1}, \textbf{g}_k - \textbf{g}_{k-1} \rangle},
\]
where $\textbf{g}_k$ denotes the the gradient \(\dfrac{\partial \mathcal{L}}{\partial \textbf{B}}\), at iteration $k$, and $\mathbf{d}_{k-1}$ is the previous search direction. The update of $\mathbf{B}$ at each iteration follows a steepest-descent step,
\[
\mathbf{B}_{k+1} = \mathbf{B}_k + M_k \textbf{d}_k
\]
where the steplength \(M_k > 0\) is chosen via a line-search procedure that enforces the Armijo sufficient-decrease condition:
\[
J(\mathbf{B}_k + M_k\mathbf{d}_{k})\leq J(\mathbf{B}_k) + M_k\delta\langle\nabla_{\mathbf{B}} J(\mathbf{B}_{k}), \mathbf{d}_k\rangle 
= J(\mathbf{B}_k) - M_k\delta\lVert\nabla_{\mathbf{B}}J(\mathbf{B}_{k})\rVert^2
\]
for $\delta \in (0,0.5)$.
The NCG algorithm terminates when the relative discrete \(l^2\)-norm between successive iterates falls below a prescribed tolerance or when a maximum number of iterations is reached. 

\subsection{Reconstruction of shear modulus function using convolutional neural networks}
\label{sec-CNN}
Convolutional neural networks (CNNs) are effective at reconstructing parameters of ill-posed PDE-constrained inverse problems. CNNs inherently perform a form of regularization through their architecture and training process. Regularization in optimization-based methods is explicitly added as penalty terms like $L^2-$norm, $L^1-$norm or Tikhonov regularization where these terms often rely on predefined assumptions about the solution such as smoothness or sparsity which may not entirely capture the intricacies of the actual problem. Consequently, these methods struggle when the solution does not conform to these predefined priors. Since CNNs are data-driven in nature that learn the relationship between inputs and outputs directly from the data, this allows them to capture complex problem-specific priors that are difficult to define explicitly in traditional methods. CNNs excel at handling nonlinear mappings through the combination of non-linear activation functions and hierarchical feature extraction, enabling them to model complex, nonlinear relationships between inputs and outputs. In a CNN model, each convolutional layer applies a convolution, that is, a linear operation, followed by a nonlinear activation function like Rectified Linear Unit (RelU) in our case. This combination of successive convolutional layers with nonlinear activations allow CNNs to approximate highly complex, nonlinear functions.

\section{Reconstruction of complex-valued shear modulus function in two dimensions}

In this section, we investigate the reconstruction of spatially varying complex-valued shear modulus functions in two dimensions. We consider the scalar modified Stokes model
\begin{equation}
    \nabla\cdot\left(A(x,y)\nabla u(x,y)\right)
    +\rho\omega^2 u(x,y)=0,
    \qquad (x,y)\in\Omega,
    \label{eq:2D_forward}
\end{equation}
where
$u$ denotes the complex-valued displacement field and $A$ is the complex-valued shear modulus. 

The coefficient is written as
\begin{equation}
    A(x,y)
    =
    (1+i)+B(x,y),
    \qquad
    B(x,y)=B_1(x,y)+iB_2(x,y),
    \label{eq:2D_complex_coeff}
\end{equation}
where $B_1$ and $B_2$ describe spatial perturbations of the real and imaginary parts of the homogeneous background coefficient. To increase the diversity of the coefficient class, we consider two different families of perturbations: Gaussian inclusions and  mollified elliptical disc inclusions.

\subsection{Coefficient models}

\subsubsection*{Gaussian inclusions}

A Gaussian perturbation is defined by
\begin{equation}
    G(x,y;c_x,c_y,\sigma,a)
    =
    a\exp\left(
    -\frac{(x-c_x)^2+(y-c_y)^2}{2\sigma^2}
    \right),
    \label{eq:2D_gaussian}
\end{equation}
where $(c_x,c_y)$ is the center, $\sigma$ determines the spatial spread and $a$ is the amplitude.

For a Gaussian sample, the real and imaginary perturbations are generated independently as
\[
B_1(x,y)
=
G(x,y;c_x^{(r)},c_y^{(r)},\sigma_r,a_r),
\]
and
\[
B_2(x,y)
=
G(x,y;c_x^{(i)},c_y^{(i)},\sigma_i,a_i).
\]

\subsubsection*{Mollified elliptical disc-shaped inclusions}

To introduce non-Gaussian and anisotropic structures, we also consider smoothly mollified elliptical discs. For a center $(c_x,c_y)$ and rotation angle $\theta$, define the rotated coordinates
\begin{align}
    \widetilde{x}
    &=
    \cos\theta\,(x-c_x)
    +
    \sin\theta\,(y-c_y),
    \\
    \widetilde{y}
    &=
    -\sin\theta\,(x-c_x)
    +
    \cos\theta\,(y-c_y).
\end{align}
The corresponding elliptical radius is
\begin{equation}
    r_E(x,y)
    =
    \sqrt{
    \left(\frac{\widetilde{x}}{r_x}\right)^2
    +
    \left(\frac{\widetilde{y}}{r_y}\right)^2
    },
\end{equation}
where $r_x$ and $r_y$ denote the two semi-axis lengths.

We define the mollified elliptical disc by
\begin{equation}
    E(x,y)
    =
    \frac{a}{2}
    \left[
    1-
    \tanh\left(
    \frac{r_E(x,y)-1}{\varepsilon}
    \right)
    \right],
    \label{eq:mollified_ellipse}
\end{equation}
where $a$ denotes the amplitude and $\varepsilon>0$ controls the width of the smooth transition across the elliptical boundary. As in the Gaussian case, the parameters of the real and imaginary components are generated independently.

The mollified elliptical disc provides a coefficient profile with a relatively flat interior and a localized smooth transition near its boundary. It therefore represents a qualitatively different spatial structure from the Gaussian inclusions used in the same training set.

\subsection{Forward problem and numerical discretization}

The computational domain $\Omega=[0,2]\times [0,2]$ is discretized uniformly using $47$ intervals in each spatial direction, resulting in a $48\times48$ grid along both the x and y-axes. The material density, $\rho$, was set to 1 $kg/m^3$ and the angular frequency, $\omega$, was fixed at 1$Hz$.

Equation~\eqref{eq:2D_forward} is discretized using a central finite-difference scheme. In particular, the coefficient is evaluated at cell interfaces by arithmetic averaging. For example,
\[
A_{i+\frac12,j}
=
\frac{A_{i,j}+A_{i+1,j}}{2},
\qquad
A_{i,j+\frac12}
=
\frac{A_{i,j}+A_{i,j+1}}{2}.
\]
The resulting sparse linear system is solved to obtain the complex-valued displacement field $u$ associated with each realization of $A$.

Complex-valued Dirichlet data are imposed on all four boundaries. With $\omega=1$, the boundary values used in the data-generation experiments are
\begin{align}
u(0,y)
&=
i e^{-i2\pi\omega y},
&
u(2,y)
&=
i e^{-i2\pi\omega y},
\\
u(x,0)
&=
i e^{-i2\pi\omega x},
&
u(x,2)
&=
i e^{-i2\pi\omega x}.
\end{align}
In physical terms, these oscillatory boundary excitations model mechanical actuators vibrating the tissue edges at angular frequency \(\omega=1\), generating shear waves propagating through viscoelastic liver tissues. 

Representative examples of the two coefficient families used in the numerical experiments are shown in Figures~\ref{fig:2D_gsn_A}--\ref{fig:2D_elps_u}. 
Figures~\ref{fig:2D_gsn_A} and~\ref{fig:2D_gsn_u} show a complex-valued shear modulus containing Gaussian perturbations and the corresponding displacement field obtained from the forward PDE, respectively. 
Similarly, Figures~\ref{fig:2D_elps_A} and~\ref{fig:2D_elps_u} show a complex-valued shear modulus containing  mollified elliptical disc-shaped perturbations and its corresponding displacement field.

\begin{figure}[H]
\centering
\includegraphics[width=0.6\linewidth]{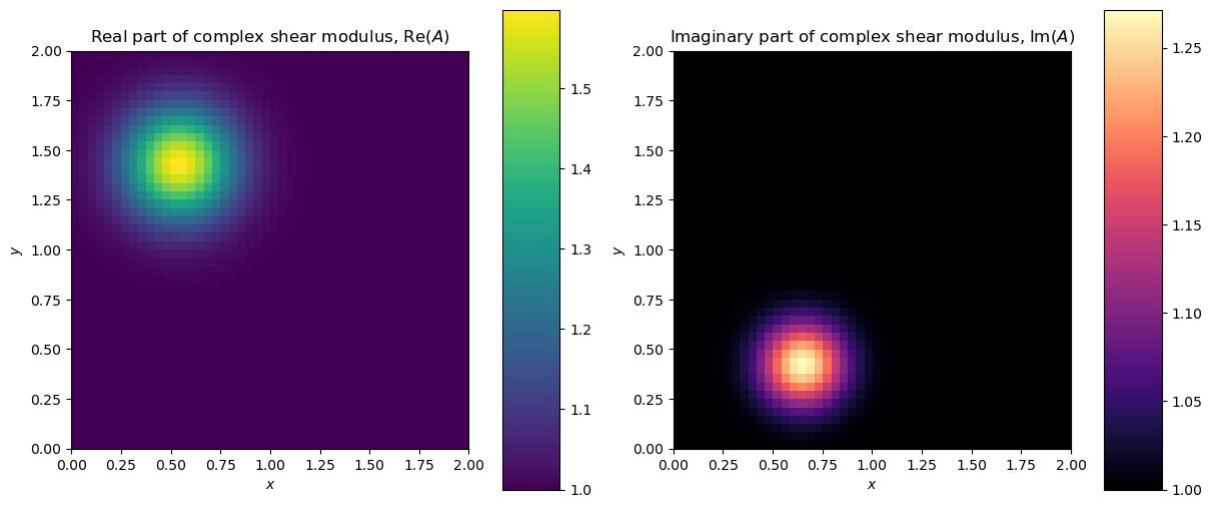}
\caption{Representative complex-valued shear modulus \(A(x,y)\) containing Gaussian perturbations in its real and imaginary components.}
\label{fig:2D_gsn_A}
\end{figure}

\begin{figure}[H]
\centering
\includegraphics[width=0.6\linewidth]{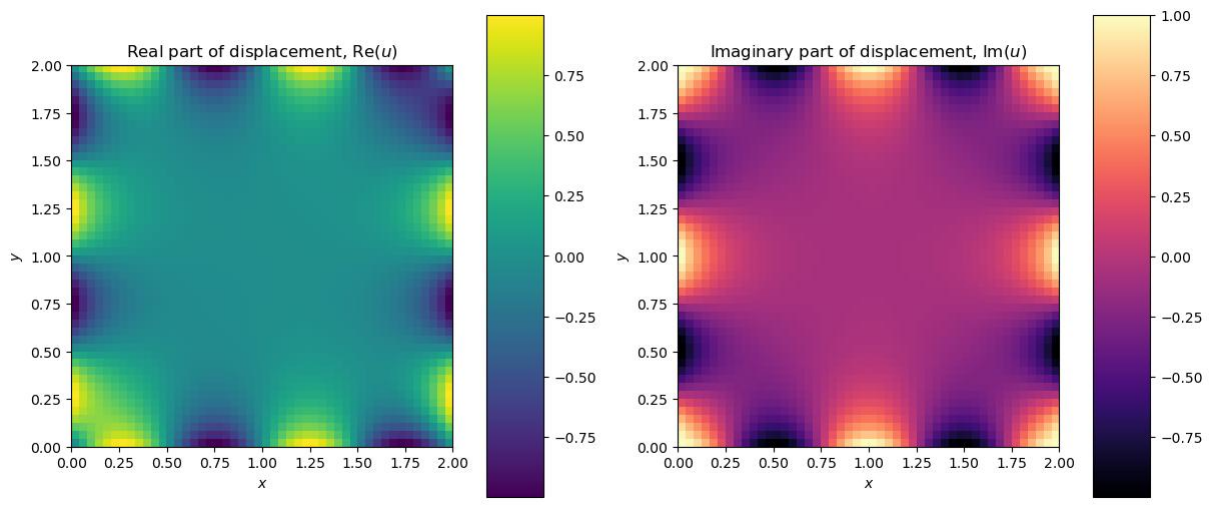}
\caption{Real and imaginary parts of the displacement field \(u(x,y)\) obtained by solving the forward PDE for the Gaussian shear modulus shown in Figure~\ref{fig:2D_gsn_A}.}
\label{fig:2D_gsn_u}
\end{figure}

\begin{figure}[H]
\centering
\includegraphics[width=0.6\linewidth]{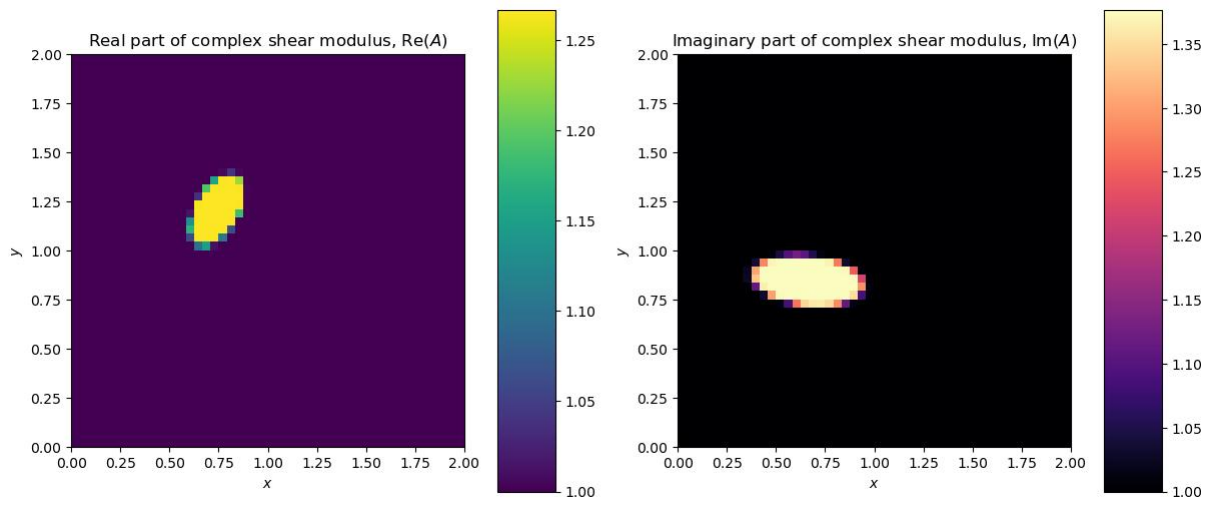}
\caption{Representative complex-valued shear modulus \(A(x,y)\) containing  mollified elliptical disc-shaped perturbations in its real and imaginary components.}
\label{fig:2D_elps_A}
\end{figure}

\begin{figure}[H]
\centering
\includegraphics[width=0.6\linewidth]{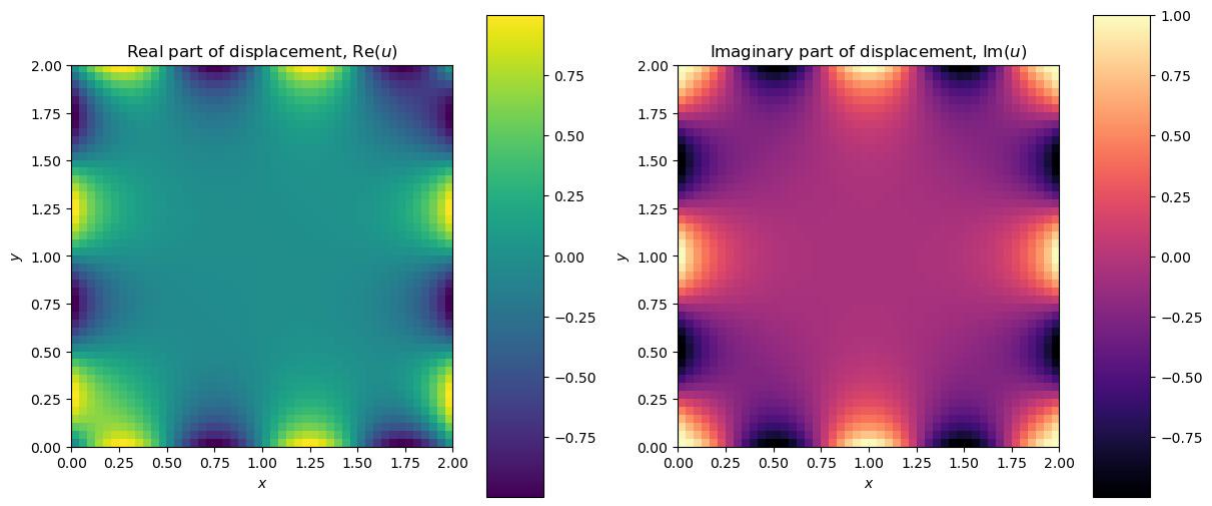}
\caption{Real and imaginary parts of the displacement field \(u(x,y)\) obtained by solving the forward PDE for the  mollified elliptical disc-shaped shear modulus shown in Figure~\ref{fig:2D_elps_A}.}
\label{fig:2D_elps_u}
\end{figure}
For every generated coefficient $A$, the corresponding solution $u$ of the forward problem is stored and used as the input to the neural reconstruction model.

\subsubsection{The inverse problem}

We now consider the inverse problem of reconstructing the spatially varying complex-valued shear modulus \(A(x,y)\), or equivalently the perturbation \(B(x,y)\), from the corresponding internal displacement measurements \(u(x,y)\). This constitutes an ill-posed inverse problem, for which small perturbations in the measured displacement data may lead to significant errors in the reconstructed coefficient. Consequently, obtaining an accurate and stable reconstruction requires an appropriate reconstruction strategy and regularization. Both convolutional neural networks and optimization-based methods provide useful approaches for solving this inverse problem, but each has its own computational challenges. A CNN can learn a direct mapping from displacement measurements to the unknown coefficient and, once trained, can generate predictions very efficiently. However, obtaining highly accurate reconstructions using a CNN alone may require a sufficiently large and representative training dataset together with extensive network training, both of which can be computationally expensive. On the other hand, PDE-constrained optimization incorporates the governing physical model directly into the reconstruction and allows regularization to be introduced explicitly. Its performance, however, can depend strongly on the choice of the initial guess. Constructing an initial coefficient sufficiently close to the desired solution is particularly difficult when the unknown is a spatially varying complex-valued function over the entire computational domain. Motivated by these complementary advantages and limitations, we use a hybrid reconstruction framework that combines a two-dimensional CNN with PDE-constrained optimization. Rather than training the CNN with the objective of obtaining the final high-accuracy reconstruction, we use it to generate an approximate spatial representation of the unknown coefficient. This allows the network to serve primarily as a data-driven initialization mechanism. The CNN prediction is subsequently supplied to the optimization procedure, where the coefficient is further refined using the governing PDE, the measured displacement data, and appropriate regularization. More precisely, given displacement data \(u_d\), the trained CNN produces \[ A_{\mathrm{CNN}} = \mathcal{N}_{\theta}(u_d), \] where \(\mathcal{N}_{\theta}\) denotes the trained neural network. Since \[ A=(1+i)+B, \] the corresponding perturbation is obtained as \[ B_{\mathrm{CNN}} = A_{\mathrm{CNN}}-(1+i). \] The optimization is then initialized with \[ B^{(0)}=B_{\mathrm{CNN}}, \] and reconstructs the perturbation \(B\) over the full computational domain. The final shear modulus is subsequently recovered from \[ A_{\mathrm{opt}}=(1+i)+B_{\mathrm{opt}}. \] The hybrid approach therefore assigns different roles to the two components of the reconstruction. The CNN provides a computationally efficient initial estimate that captures the approximate locations, geometries, and amplitudes of the inclusions, without requiring the network itself to achieve the desired final accuracy. The PDE-constrained optimization then fine-tunes this prediction by enforcing consistency with the forward model and incorporating regularization. In this way, the proposed framework exploits the rapid data-driven prediction capability of the CNN while reducing its training burden, and simultaneously alleviates the initialization difficulty associated with optimization-based reconstruction.

\subsection{Hybrid reconstruction framework based on convolutional neural network and optimization}

The proposed hybrid reconstruction procedure consists of two successive stages. In the first stage, a two-dimensional convolutional neural network is trained to learn an approximate inverse mapping from the complex-valued displacement field to the corresponding complex-valued coefficient field. In the second stage, the CNN prediction is used as the initial guess for a PDE-constrained optimization algorithm.
The CNN therefore provides a computationally inexpensive approximation of the unknown coefficient, while the optimization stage refines this approximation using the governing physical model.
A principal advantage of this formulation is that the CNN is not required to reproduce the unknown coefficient with the accuracy expected from the final reconstruction. Instead, it only needs to provide an initial coefficient field sufficiently close to the desired solution for the optimization procedure to perform an effective refinement. This reduces the dependence of the optimization algorithm on arbitrary initial guesses while retaining the PDE model in the final reconstruction stage.

\subsubsection{CNN-based initialization}

Unlike the parameter-based CNN reconstruction considered previously in the one-dimensional numerical experiments, the two-dimensional CNN used here predicts the complete coefficient field directly. It does not estimate a prescribed set of geometric parameters such as the center, amplitude, spread, or orientation of an inclusion. Instead, the network provides a spatially resolved approximation of the complex-valued shear modulus over the entire computational domain.

Since both the displacement \(u\) and the coefficient \(A\) are complex-valued, their real and imaginary parts are supplied to the network as separate channels. The CNN input therefore has the form
\[
\left(
\operatorname{Re}(u),
\operatorname{Im}(u)
\right)
\in
\mathbb{R}^{48\times48\times2},
\]
while the corresponding training target is
\[
\left(
\operatorname{Re}(A),
\operatorname{Im}(A)
\right)
\in
\mathbb{R}^{48\times48\times2}.
\]

After training, a measured displacement field \(u_d\) is passed through the CNN to obtain
\[
A_{\mathrm{CNN}}
=
\operatorname{Re}(A_{\mathrm{CNN}})
+
i\,\operatorname{Im}(A_{\mathrm{CNN}}).
\]
This field is not regarded as the final solution of the inverse problem. Rather, it serves as the initialization for the optimization-based reconstruction described later.

\subsubsection{Training dataset}

To train the CNN on a broader class of spatial coefficient distributions, we generated a mixed dataset consisting of complex-valued Gaussian and  mollified elliptical disc-shaped perturbations. These two families were chosen to expose the CNN to coefficient distributions with qualitatively different spatial features. In particular, the Gaussian perturbations provide smoothly varying profiles with gradual spatial transitions, whereas the  mollified elliptical discs contain a relatively flat interior region together with localized transitions near the boundary of the inclusion. Thus, the mixed training set allows the network to learn from both continuously varying profiles and plateau-like structures.

The training dataset contains \(4000\) coefficient--displacement pairs. Half of the samples contain Gaussian perturbations and the remaining half contain  mollified elliptical disc-shaped perturbations, with the two sample families randomly shuffled before training.

For a Gaussian sample, the real and imaginary components are generated independently according to

$$
B_j(x,y)
=
a_j
\exp\left(
-\frac{(x-c_{x,j})^2+(y-c_{y,j})^2}
{2\sigma_j^2}
\right),
\qquad j=1,2,
$$

where the amplitudes are sampled from

$$
a_j\in(0.1,1.0),
$$

the standard deviations from

$$
\sigma_j\in(0.03,0.5),
$$

and the centers are randomly distributed over the computational domain, \([0,2]\times[0,2]\).

For the  mollified elliptical disc-shaped samples, the amplitudes are again sampled from \((0.1,1.0)\). The semi-axis lengths are randomly sampled from

$$
r_x\in(0.12,0.40),
\qquad
r_y\in(0.08,0.35),
$$

the rotation angle from

$$
\theta\in(0,\pi),
$$

and the mollification parameter from

$$
\varepsilon\in(0.03,0.08).
$$

The centers of the elliptical discs are sampled with an appropriate margin from the boundary so that most of each inclusion remains inside the computational domain, \([0,2]\times[0,2]\).

For every realization of \(A\), the corresponding complex-valued displacement \(u\) is obtained by solving the forward PDE. The resulting pair

$$
\left(u,A\right)
$$

is then used as one training sample for the CNN.

\subsubsection{CNN architecture and training}

We employ an encoder--decoder convolutional neural network to map the
complex-valued displacement field to an initial reconstruction of the
complex-valued coefficient. The real and imaginary components of the
displacement are supplied as two input channels,
\[
\bigl(\operatorname{Re}(u),\operatorname{Im}(u)\bigr)
\in \mathbb{R}^{48\times48\times2},
\]
and the network produces two spatial output fields corresponding to
$\operatorname{Re}(A_{\mathrm{CNN}})$ and
$\operatorname{Im}(A_{\mathrm{CNN}})$.

The encoder progressively reduces the spatial resolution from
$48\times48$ to $12\times12$ while increasing the number of feature
channels from $32$ to $128$. The decoder subsequently restores the
original spatial resolution through two upsampling stages. The resulting
shared representation is passed through separate output branches for the
real and imaginary components of the coefficient. All convolutional
layers use $3\times3$ kernels with same padding and ReLU activations,
except for the final output layers, which use linear activations. The
network contains $444226$ trainable parameters. The complete architecture
is summarized in Table~\ref{tab:cnn_architecture}.

\begin{table}[H]
\centering
\caption{Architecture of the 2D convolutional neural network used to
generate the initial coefficient reconstruction.}
\label{tab:cnn_architecture}

\begin{tabular}{lll}
\hline
\textbf{Stage} & \textbf{Operation} & \textbf{Output size} \\
\hline

Input
& --
& $48\times48\times2$ \\

Encoder 1
& Conv2D $(3\times3)$, 32 filters, ReLU
& $48\times48\times32$ \\

Downsampling 1
& Conv2D $(3\times3)$, 32 filters, stride 2, ReLU
& $24\times24\times32$ \\

Encoder 2
& Conv2D $(3\times3)$, 64 filters, ReLU
& $24\times24\times64$ \\

Downsampling 2
& Conv2D $(3\times3)$, 64 filters, stride 2, ReLU
& $12\times12\times64$ \\

Latent representation
& 2 Conv2D $(3\times3)$, 128 filters, ReLU
& $12\times12\times128$ \\

Decoder 1
& Upsampling $(2\times2)$ + 2 Conv2D $(3\times3)$, 64 filters
& $24\times24\times64$ \\

Decoder 2
& Upsampling $(2\times2)$ + 2 Conv2D $(3\times3)$, 32 filters
& $48\times48\times32$ \\

Real branch
& Conv2D, 32 filters + Conv2D, 1 filter, linear
& $48\times48\times1$ \\

Imaginary branch
& Conv2D, 32 filters + Conv2D, 1 filter, linear
& $48\times48\times1$ \\

\hline
\end{tabular}
\end{table}

The network is trained using the Adam optimizer with a mean-squared-error
loss. The $4000$ coefficient--displacement pairs are divided using an
$80/20$ training--validation split, giving $3200$ training and $800$
validation samples. Training is performed for $300$ epochs with a batch
size of $32$, with the training samples shuffled at each epoch.

A typical training epoch required approximately $175$--$190$ seconds on
the computational platform used in this study. Preliminary experiments
indicated that further improvement of the CNN reconstruction required
approximately $500$--$600$ epochs. Since the CNN serves only to provide
an informative initialization for the subsequent PDE-constrained
optimization, we restrict training to $300$ epochs and use the
optimization stage to refine the resulting coefficient estimate.
\subsection{Hybrid CNN--optimization reconstruction}

The CNN provides a computationally inexpensive reconstruction of the
complete complex-valued shear-modulus field. However, since the CNN
prediction is obtained directly from the measured displacement data,
it is not explicitly constrained to reproduce the measurements through
the governing PDE. Moreover, errors may remain in the reconstructed
amplitudes, inclusion boundaries, and local spatial structure. We
therefore use the CNN reconstruction as an initial guess for a
PDE-constrained optimization method.

Let
\[
A=(1+i)+B,
\]
where $B$ represents the spatially varying perturbation of the
background coefficient. Given a coefficient $B$, let $u(B)$ denote
the solution of the forward problem
\[
\nabla\cdot\left(A\nabla u\right)+\rho\omega^2u=0
\qquad\text{in }\Omega,
\]
subject to the prescribed Dirichlet boundary condition. The measured
internal displacement is denoted by $u_d$.

Starting from the CNN prediction, we minimize the regularized
data-misfit functional
\begin{equation}
J(B)
=
\frac{\alpha}{2}
\int_\Omega
|u(B)-u_d|^2\,dx
+
\frac{\beta}{2}
\int_\Omega
|B|^2\,dx,
\label{eq:hybrid_objective}
\end{equation}
where
\[
B^{(0)}=B_{\mathrm{CNN}},
\]
Thus, the CNN reconstruction serves two purposes: it provides the
initial iterate for the optimization and acts as a reference in the
regularization term.

The first term in~\eqref{eq:hybrid_objective} measures the discrepancy
between the displacement generated by the reconstructed coefficient
and the measured internal displacement. The second term represents
a zero-order Tikhonov regularization term. The regularization term
\[
\frac{\beta}{2}\int_{\Omega}|B|^2\,dx.
\]
penalizes large values of the coefficient perturbation and
provides regularization for the ill-posed inverse problem. The
regularization parameter $\beta>0$ controls the balance between fitting
the measured displacement data and controlling the magnitude of the
reconstructed perturbation.

The CNN prediction is used only to initialize the optimization,
\[
B^{(0)}=B_{\mathrm{CNN}},
\]
and does not enter the regularization term. Thus, the CNN provides an
informed starting point for the iterative reconstruction, while the
subsequent PDE-constrained optimization, together with Tikhonov
regularization, refines the coefficient using the measured displacement
data and the governing forward model.

\subsubsection{Adjoint-based gradient computation}

Since the displacement $u(B)$ depends implicitly on the coefficient
through the forward PDE, direct differentiation of the objective with
respect to every coefficient value would be computationally expensive.
We therefore compute the gradient using a discrete adjoint formulation.
At each optimization iteration, the forward problem is first solved for
the current coefficient $B^{(k)}$. An associated adjoint problem is
then solved using the displacement residual $u(B^{(k)})-u_d$.

The resulting adjoint variables are used to evaluate the derivatives
of~\eqref{eq:hybrid_objective} with respect to the real and imaginary
components of $B$. Both components are optimized simultaneously, so
that the optimization updates the complete complex-valued coefficient.

\subsubsection{Nonlinear conjugate-gradient iteration}

The minimization of~\eqref{eq:hybrid_objective} is performed using a
nonlinear conjugate-gradient (NCG) method. Starting with
\[
B^{(0)}=B_{\mathrm{CNN}},
\]
the initial search direction is chosen as the negative gradient,
\[
d^{(0)}=-\nabla J(B^{(0)}).
\]
For subsequent iterations, the search direction combines the current
negative gradient with the previous search direction,
\[
d^{(k)}
=
-\nabla J(B^{(k)})
+
\gamma_k d^{(k-1)},
\]
where $\gamma_k$ is the nonlinear conjugate-gradient parameter.

The coefficient is updated according to
\[
B^{(k+1)}
=
B^{(k)}+s_k d^{(k)},
\]
where the step length $s_k$ is determined using an Armijo
backtracking line search.

\subsubsection{Armijo line search}

For a given search direction $d^{(k)}$, the line search begins with a
trial step length and successively reduces it until the sufficient
decrease condition
\[
J\!\left(B^{(k)}+s_kd^{(k)}\right)
\leq
J(B^{(k)})
+
c\,s_k
\operatorname{Re}
\left\langle
\nabla J(B^{(k)}),d^{(k)}
\right\rangle
\]
is satisfied, where $c>0$ is the Armijo parameter. This adaptive
choice of the step length prevents excessively large coefficient
updates while allowing the optimization to make larger corrections
when they produce sufficient reduction of the objective.

The optimization is terminated when the prescribed stopping criterion
is satisfied or when the maximum number of iterations is reached.
The resulting hybrid reconstruction therefore combines the learned
spatial information contained in the CNN prediction with refinement
based explicitly on the governing PDE and the measured displacement.

\subsection{Numerical experiments}
We investigate the generalization capability of the proposed framework through two sets of numerical experiments. First, we consider Gaussian
and  mollified elliptical disc-shaped inclusions separately and evaluate the CNN and hybrid reconstructions on test samples of the same structural classes
used during training. Second, we consider a more challenging out-of-distribution configuration in which Gaussian and  mollified elliptical disc structures occur simultaneously in the same coefficient
field, although such combined configurations are not included in the training data.

\subsubsection{Experiment I: reconstruction of individual shape families}

The first test dataset contains $500$ samples, consisting of
\[
250 \text{ Gaussian samples}
\qquad\text{and}\qquad
250 \text{ mollified-ellipse samples}.
\]
As in the training dataset, each sample belongs to only one of the two shape families.

For the Gaussian test samples, the amplitudes are sampled from
\[
a_r,a_i\sim\mathcal{U}(0.2,0.8),
\]
with
\[
\sigma_r,\sigma_i\sim\mathcal{U}(0.1,0.3),
\]
and center coordinates sampled from
\[
c_x,c_y\sim\mathcal{U}(0.3,1.7).
\]

For the  mollified elliptical disc-shaped test samples, the amplitudes satisfy
\[
a_r,a_i\sim\mathcal{U}(0.2,0.8),
\]
the semi-axis parameters are sampled from
\[
r_x\sim\mathcal{U}(0.2,0.3),
\qquad
r_y\sim\mathcal{U}(0.1,0.3),
\]
with
\[
\theta\sim\mathcal{U}(0,\pi),
\qquad
\varepsilon\sim\mathcal{U}(0.03,0.08).
\]

This experiment evaluates whether a single CNN trained jointly on two coefficient families can identify and reconstruct either family directly from the corresponding internal displacement data. The test data considered in this experiment does not have noise.

To assess the performance of the CNN over the complete test dataset, we
compute the $L^2(\Omega)$ reconstruction errors for the real and imaginary
components of the coefficient, as well as for the full complex-valued
coefficient. For each test sample, these errors are defined by
\[
E_{\mathrm{Re}}
=
\left\|
\operatorname{Re}(A_{\mathrm{CNN}})
-
\operatorname{Re}(A_d)
\right\|_{L^2(\Omega)},
\]
\[
E_{\mathrm{Im}}
=
\left\|
\operatorname{Im}(A_{\mathrm{CNN}})
-
\operatorname{Im}(A_d)
\right\|_{L^2(\Omega)},
\]
and
\[
E_A
=
\left\|
A_{\mathrm{CNN}}-A_d
\right\|_{L^2(\Omega)}.
\]

Table~\ref{tab:cnn_test_errors} summarizes the reconstruction errors over
all $500$ test samples. The CNN achieves mean errors of
$6.4693\times10^{-2}$ and $6.6093\times10^{-2}$ for the real and imaginary
components, respectively, while the mean error for the complete
complex-valued coefficient is $9.4681\times10^{-2}$. The corresponding
median errors are slightly smaller than the mean errors, indicating that
the reconstruction error for the majority of the test samples is below
the overall average.

\begin{table}[H]
\centering
\caption{CNN reconstruction errors over the $500$ test samples in
Experiment I.}
\label{tab:cnn_test_errors}

\begin{tabular}{lccccc}
\hline
 & Mean & Std. & Median & Min. & Max. \\
\hline

$\operatorname{Re}(A)$
& $6.4693\times10^{-2}$
& $3.0304\times10^{-2}$
& $5.8194\times10^{-2}$
& $1.4715\times10^{-2}$
& $2.1719\times10^{-1}$ \\

$\operatorname{Im}(A)$
& $6.6093\times10^{-2}$
& $2.7720\times10^{-2}$
& $6.0846\times10^{-2}$
& $1.8937\times10^{-2}$
& $1.9950\times10^{-1}$ \\

$A$
& $9.4681\times10^{-2}$
& $3.5716\times10^{-2}$
& $8.8932\times10^{-2}$
& $3.1948\times10^{-2}$
& $2.6949\times10^{-1}$ \\

\hline
\end{tabular}
\end{table}

The distributions of these errors over the test dataset are shown in
Figure~\ref{fig:cnn_test_error_histograms}. The histograms provide a more
complete characterization of the CNN performance than a single
representative reconstruction and illustrate the variability of the
prediction accuracy across the randomly generated test coefficients.
\begin{figure}[H]
\centering

\begin{subfigure}{0.32\textwidth}
    \centering
    \includegraphics[width=\linewidth]
    {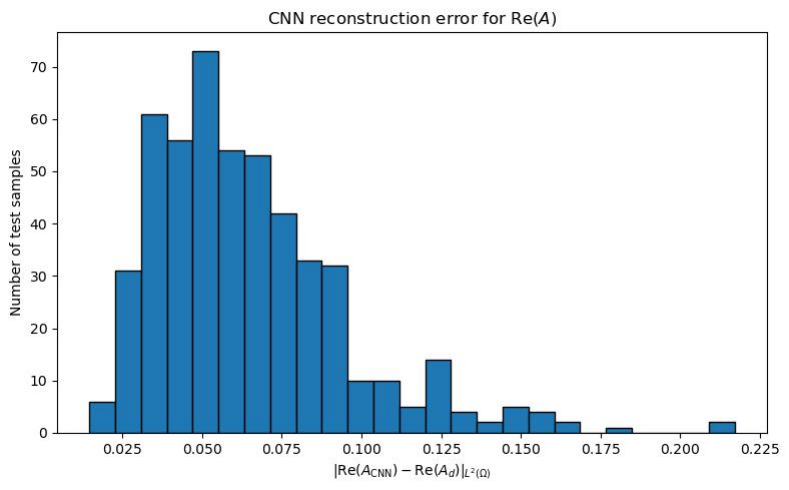}
    \caption{$L^2(\Omega)$ reconstruction errors for
    $\operatorname{Re}(A)$.}
\end{subfigure}
\hfill
\begin{subfigure}{0.32\textwidth}
    \centering
    \includegraphics[width=\linewidth]
    {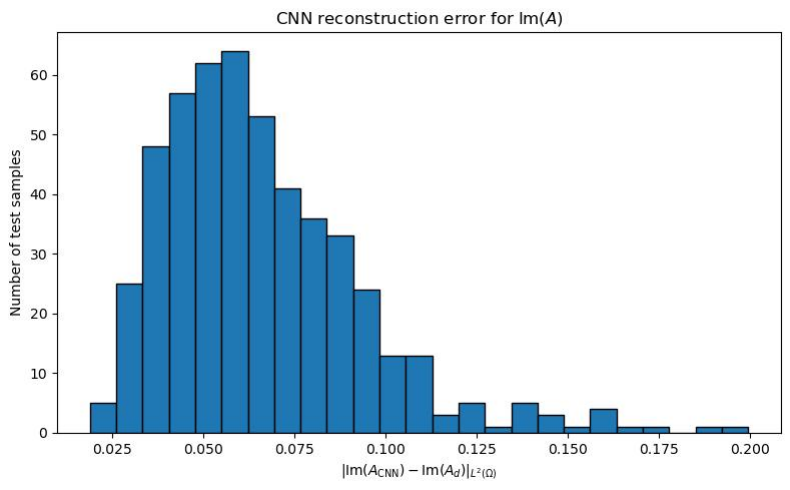}
    \caption{$L^2(\Omega)$ reconstruction errors for
    $\operatorname{Im}(A)$.}
\end{subfigure}
\hfill
\begin{subfigure}{0.32\textwidth}
    \centering
    \includegraphics[width=\linewidth]
    {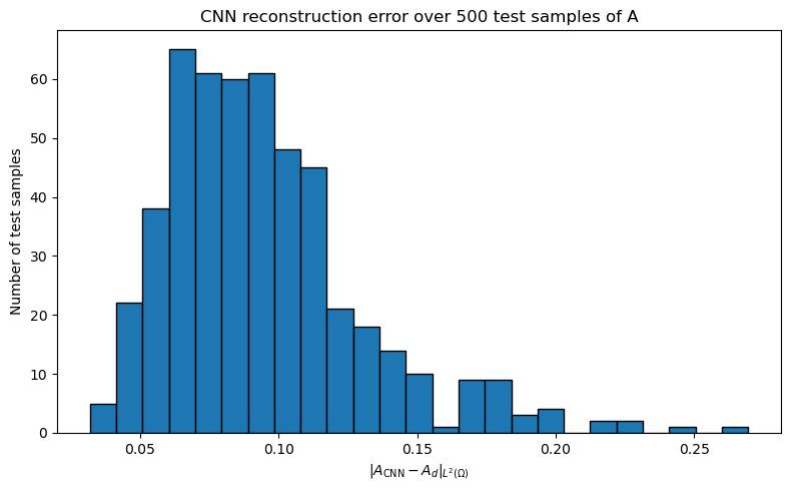}
    \caption{$L^2(\Omega)$ reconstruction errors for the
    complex-valued coefficient $A$.}
\end{subfigure}

\caption{Distribution of the CNN reconstruction errors over all $500$
test samples in Experiment I. The test dataset contains both Gaussian
and mollified-ellipse coefficient fields.}
\label{fig:cnn_test_error_histograms}

\end{figure}
Having evaluated the CNN over the complete test dataset, we next examine
representative reconstructions from the two coefficient families in more
detail. These examples illustrate both the spatial structure recovered by
the CNN and the quality of the initial coefficient fields supplied to the
subsequent PDE-constrained optimization stage.

Representative CNN predictions for the two individual shape families are shown in ~\ref{fig:CNN_ellipse_sample} and
Figures~\ref{fig:CNN_gaussian_sample}.
For each test sample, the input to the network consists of the real and
imaginary components of the corresponding complex-valued internal displacement
field, while the output is an approximation of the spatially varying
complex-valued coefficient field.

For a representative test coefficient containing a complex-valued
mollified-ellipse perturbation, the corresponding internal displacement
field is shown in Figure~\ref{fig:ellipse_displacement_input}. This
displacement field constitutes the measured data for the inverse problem, with its real and imaginary components.

\begin{figure}[H]
\centering
\includegraphics[width=0.5\linewidth]
{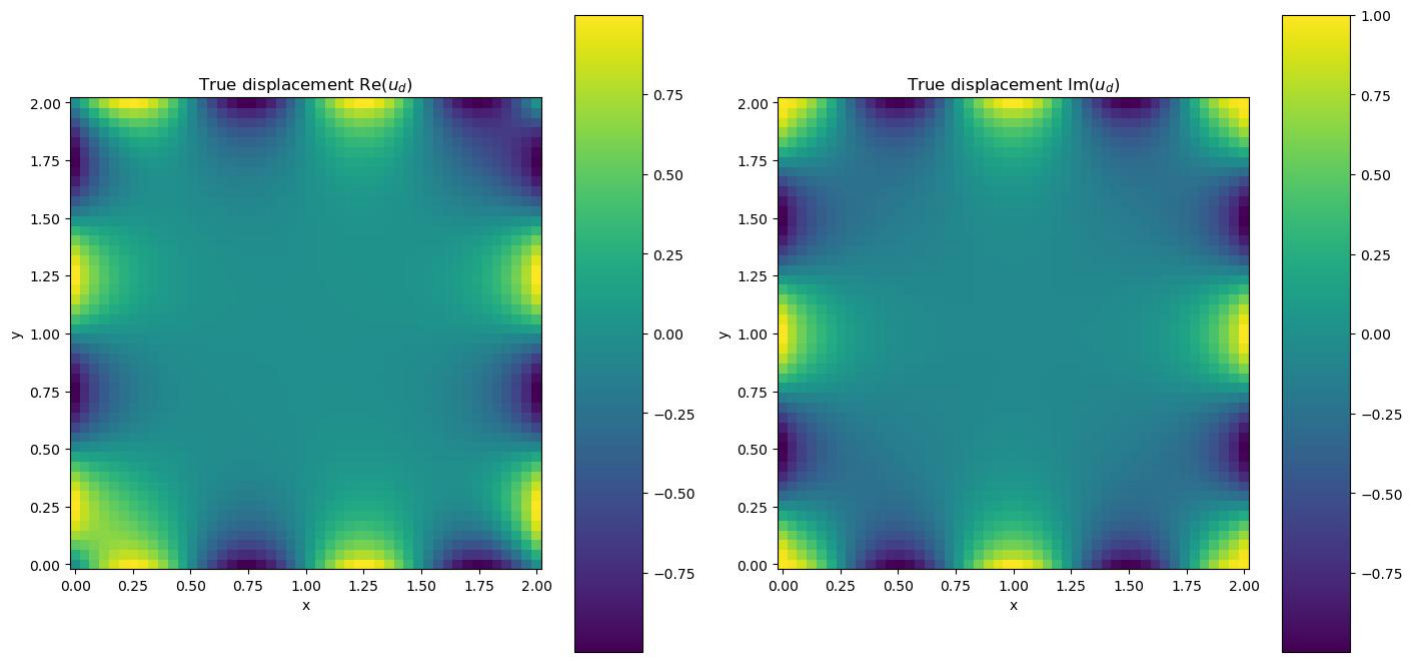}
\caption{Real and imaginary components of the internal displacement field
\(u_d\) corresponding to a representative complex-valued mollified-ellipse
perturbation. These displacement data form the two input channels to the CNN.}
\label{fig:ellipse_displacement_input}
\end{figure}

Using the displacement field in
Figure~\ref{fig:ellipse_displacement_input}, the CNN produces the coefficient
estimate shown in Figure~\ref{fig:CNN_ellipse_sample}. The prediction captures
the approximate locations and geometries of the real and imaginary
perturbations, while some discrepancies in shape and amplitude remain.

\begin{figure}[H]
\centering

\begin{subfigure}{0.5\textwidth}
    \centering
    \includegraphics[width=\linewidth]
    {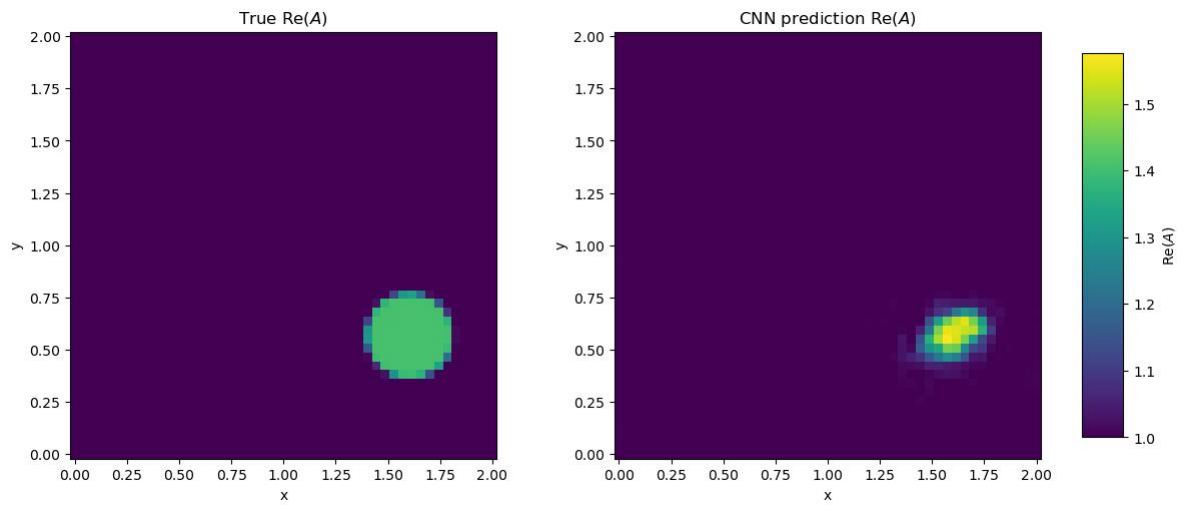}
    \caption{Real component of the coefficient.}
    \label{fig:CNN_ellipse_real}
\end{subfigure}

\begin{subfigure}{0.5\textwidth}
    \centering
    \includegraphics[width=\linewidth]
    {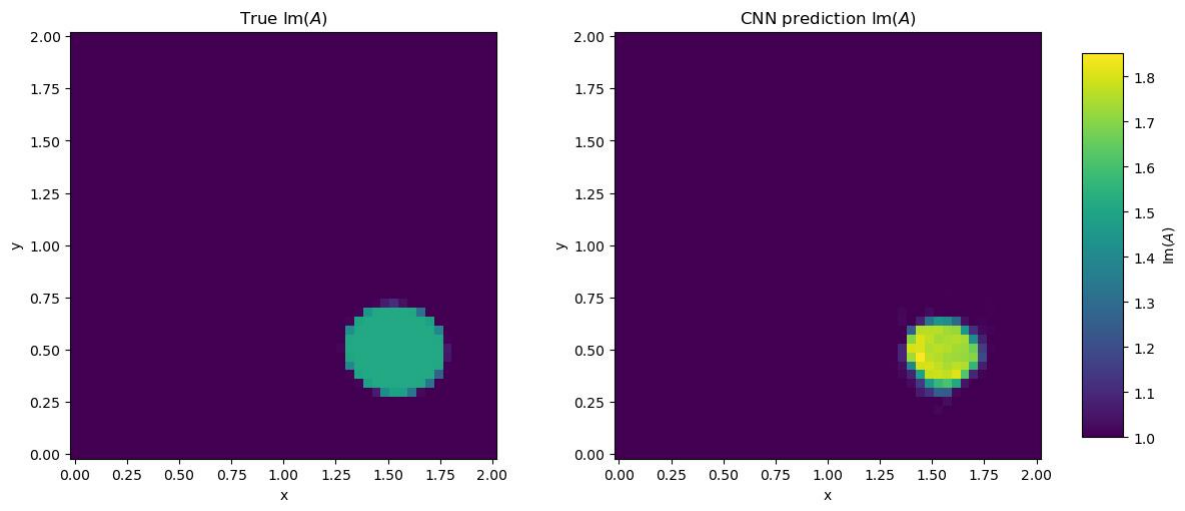}
    \caption{Imaginary component of the coefficient.}
    \label{fig:CNN_ellipse_imag}
\end{subfigure}

\caption{Representative CNN prediction for a complex-valued
mollified-ellipse test sample from Experiment I. The top row compares
the true and CNN-predicted real components of \(A(x,y)\), while the
bottom row compares the corresponding imaginary components.}

\label{fig:CNN_ellipse_sample}
\end{figure}

\begin{figure}[H]
\centering
\includegraphics[width=0.5\linewidth]
{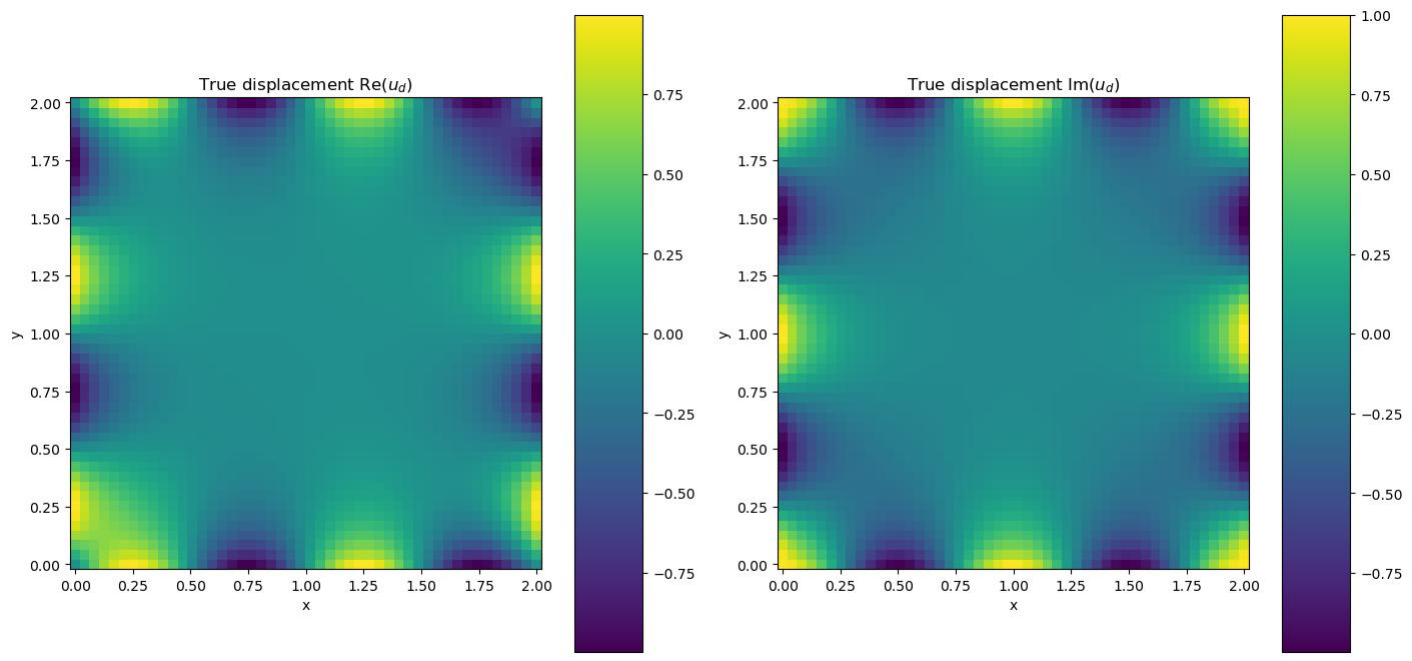}

\caption{Real and imaginary components of the internal displacement field
\(u_d\) corresponding to a representative complex-valued Gaussian
perturbation. These displacement data are supplied to the CNN as the two
input channels for predicting the corresponding complex-valued coefficient
field.}

\label{fig:gaussian_displacement_input}
\end{figure}

A representative result for the Gaussian test family is shown in
Figure~\ref{fig:CNN_gaussian_sample}, using the displacement field in
Figure~\ref{fig:gaussian_displacement_input} as input to the CNN.
\begin{figure}[H]
\centering

\begin{subfigure}{0.5\textwidth}
    \centering
    \includegraphics[width=\linewidth]
    {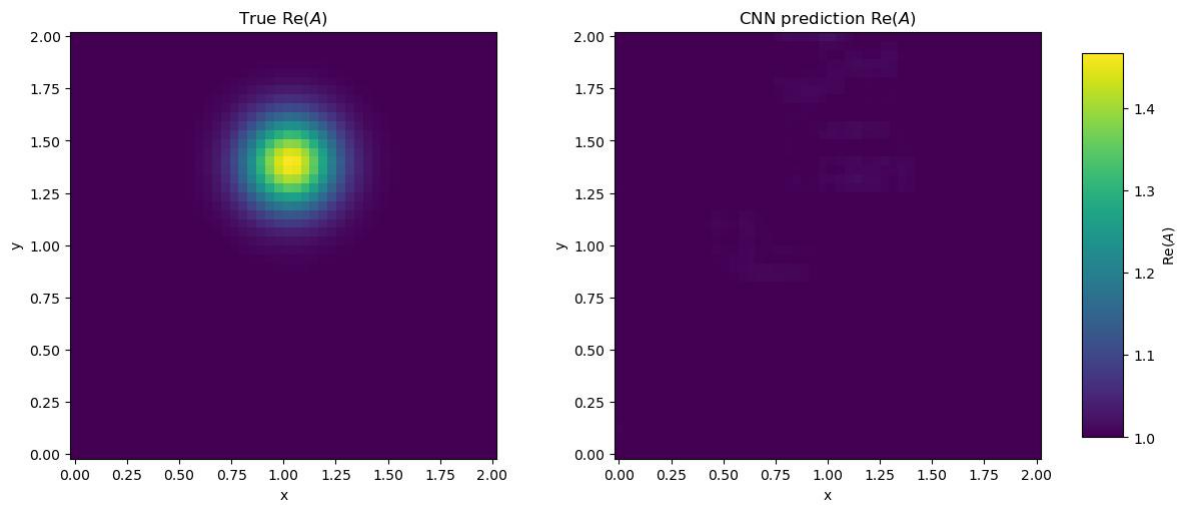}
    \caption{Real component of the coefficient.}
    \label{fig:CNN_gaussian_real}
\end{subfigure}

\begin{subfigure}{0.5\textwidth}
    \centering
    \includegraphics[width=\linewidth]
    {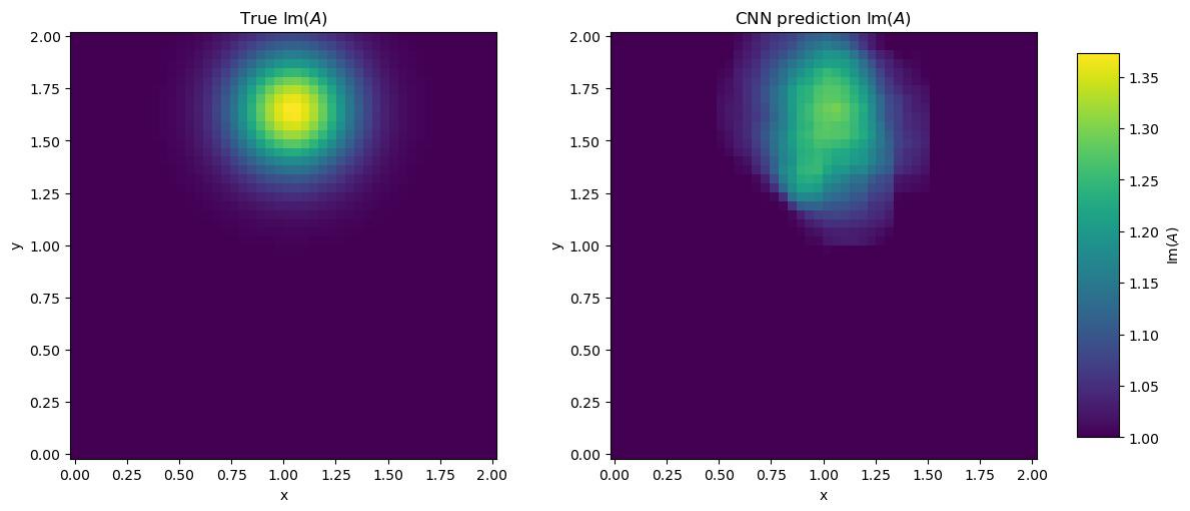}
    \caption{Imaginary component of the coefficient.}
    \label{fig:CNN_gaussian_imag}
\end{subfigure}

\caption{Representative CNN prediction for a complex-valued
Gaussian test sample from Experiment I. The top row compares
the true and CNN-predicted real components of \(A(x,y)\), while the
bottom row compares the corresponding imaginary components. The
CNN prediction is subsequently used as the initial guess for the
PDE-constrained optimization stage of the hybrid reconstruction
framework.}

\label{fig:CNN_gaussian_sample}
\end{figure}

To further illustrate the CNN prediction for the real component, we
additionally display the same prediction using its own color scale in
Figure~\ref{fig:CNN_gaussian_real_individual_scale}. Although the variation
in the predicted real component is small relative to the amplitude of the
true coefficient and is therefore barely visible on the common color scale
in Figure~\ref{fig:CNN_gaussian_sample}, the individually scaled
visualization reveals the spatial structure present in the CNN output.

\begin{figure}[H]
\centering
\includegraphics[width=0.5\textwidth]
{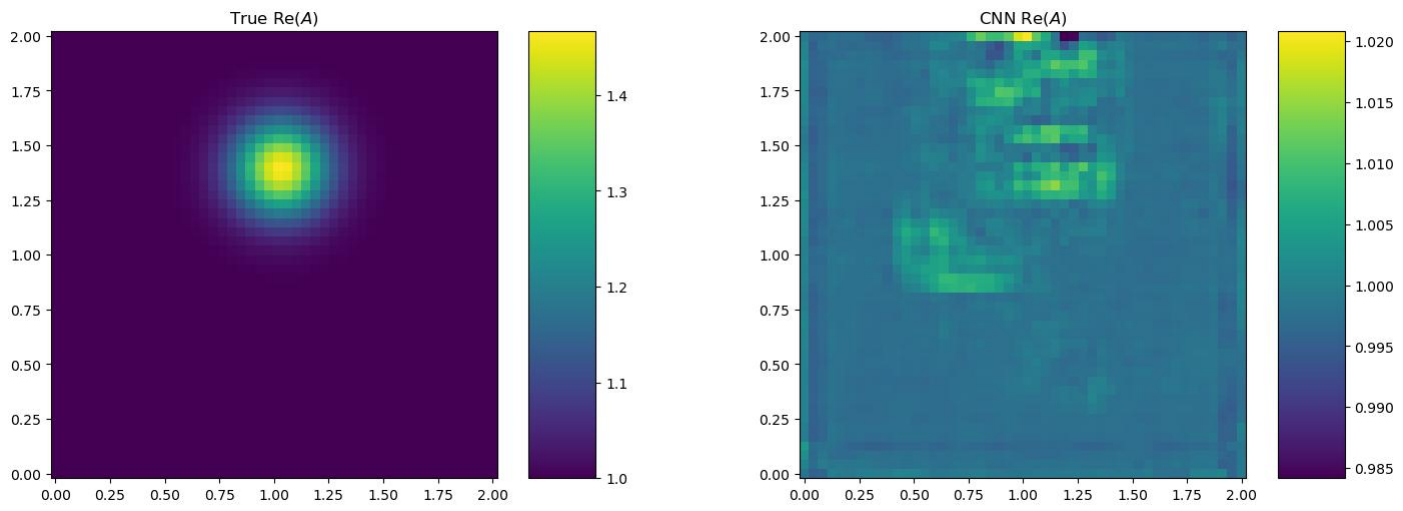}

\caption{Real component of the representative Gaussian test sample shown
using separate color scales for the true coefficient and the CNN prediction.
The individually scaled CNN prediction reveals the weak spatial variations
that are not readily visible when the common color scale in
Figure~\ref{fig:CNN_gaussian_sample} is used.}

\label{fig:CNN_gaussian_real_individual_scale}
\end{figure}

For both shape families, the CNN recovers the principal spatial features of
the unknown perturbations, including their approximate locations, amplitudes,
and geometries, although smoothing and small spatial artifacts may remain.

The CNN-predicted coefficient shown in
Figure~\ref{fig:CNN_ellipse_sample} and Figure~\ref{fig:CNN_gaussian_sample} is subsequently used as informed initial
guesses for the PDE-constrained optimization. The optimization refines the
coefficient by minimizing the discrepancy between the measured internal
displacement and the displacement generated by the forward model, together
with the prescribed regularization. 

\paragraph{Optimization setup.}

For the optimization stage of the hybrid framework, the coefficient field
predicted by the CNN is used as the initial iterate. The reconstruction is
then refined by minimizing the regularized PDE-constrained objective
functional introduced previously. The nonlinear conjugate-gradient (NCG)
method with an Armijo backtracking line search is used for the minimization.

For the representative mollified-ellipse and Gaussian test samples considered
in this experiment, the optimization parameters are summarized in
Table~\ref{tab:experiment1_optimization_parameters}.

\begin{table}[H]
\centering
\caption{Optimization parameters used for the representative
mollified-ellipse and Gaussian test samples in Experiment I.}
\label{tab:experiment1_optimization_parameters}

\begin{tabular}{lcc}
\hline
Parameter
& mollified elliptical disc
& Gaussian \\
\hline
Data-fidelity weight $\alpha$
& 20
& 20 \\

Regularization parameter $\beta$
& $1\times10^{-7}$
& $5\times10^{-6}$ \\

Maximum NCG iterations $k_{\max}$
& 1000
& 2000 \\

Initial line-search step size
& 1
& 1 \\

Armijo parameter
& $1\times10^{-4}$
& $1\times10^{-4}$ \\

Backtracking reduction factor
& 0.8
& 0.8 \\
\hline
\end{tabular}
\end{table}

To directly illustrate the contribution of the optimization stage, we compare
the true coefficient, the CNN-generated initial estimate, and the coefficient
obtained after PDE-constrained optimization for both shape families. The
corresponding comparisons for the mollified-ellipse and Gaussian test samples
are shown in Figures~\ref{fig:ellipse_hybrid_comparison} and
\ref{fig:gaussian_hybrid_comparison}, respectively. 


\begin{figure}[H]
\centering
\includegraphics[width=0.7\linewidth]
{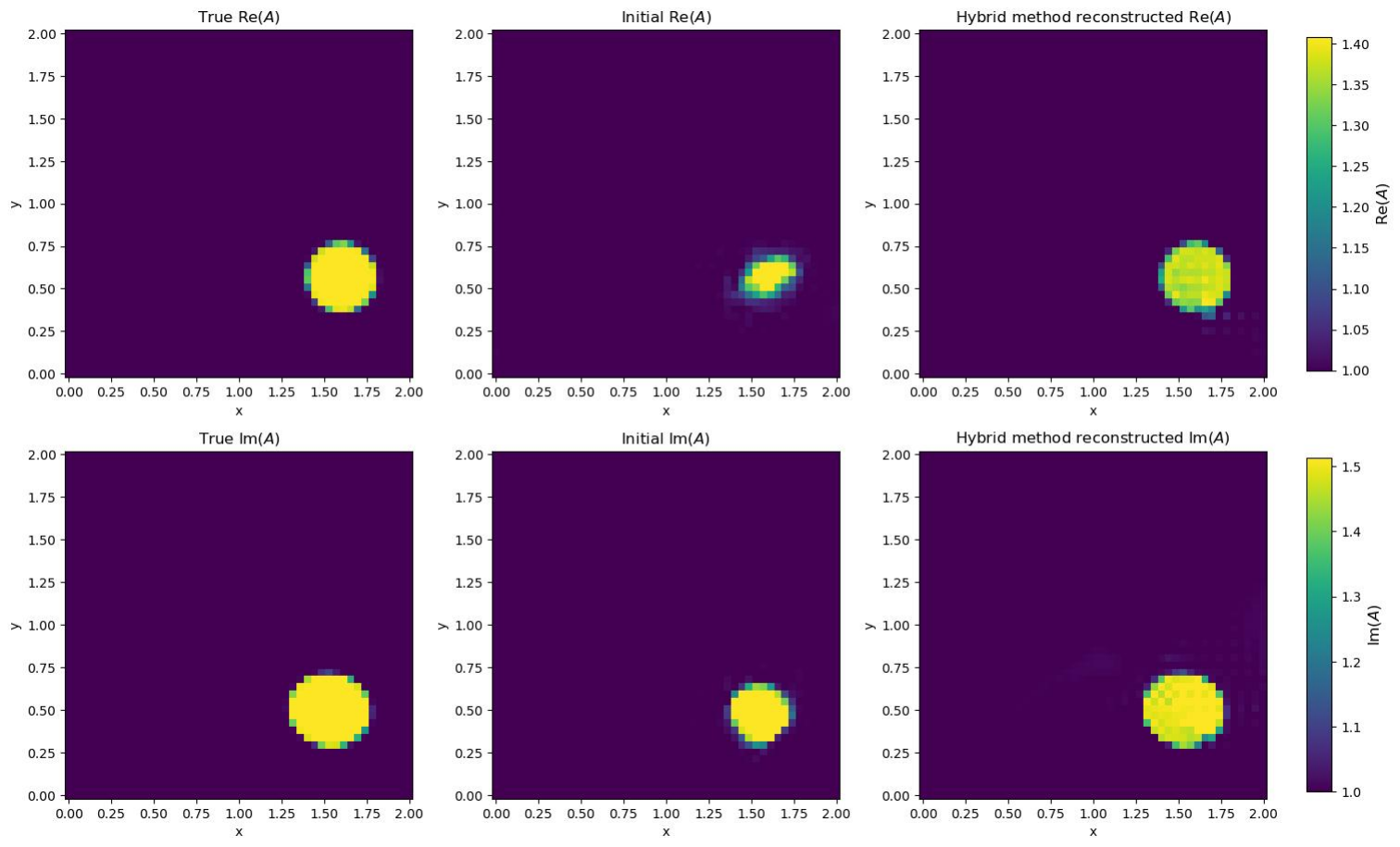}

\caption{End-to-end hybrid reconstruction for a representative
complex-valued mollified-ellipse test sample from Experiment I.
The left column shows the true coefficient, the middle column shows
the CNN-generated initial estimate, and the right column shows the
coefficient reconstructed after PDE-constrained optimization.
The top and bottom rows correspond to the real and imaginary
components of \(A(x,y)\), respectively.}

\label{fig:ellipse_hybrid_comparison}
\end{figure}


\begin{figure}[H]
\centering
\includegraphics[width=0.7\linewidth]
{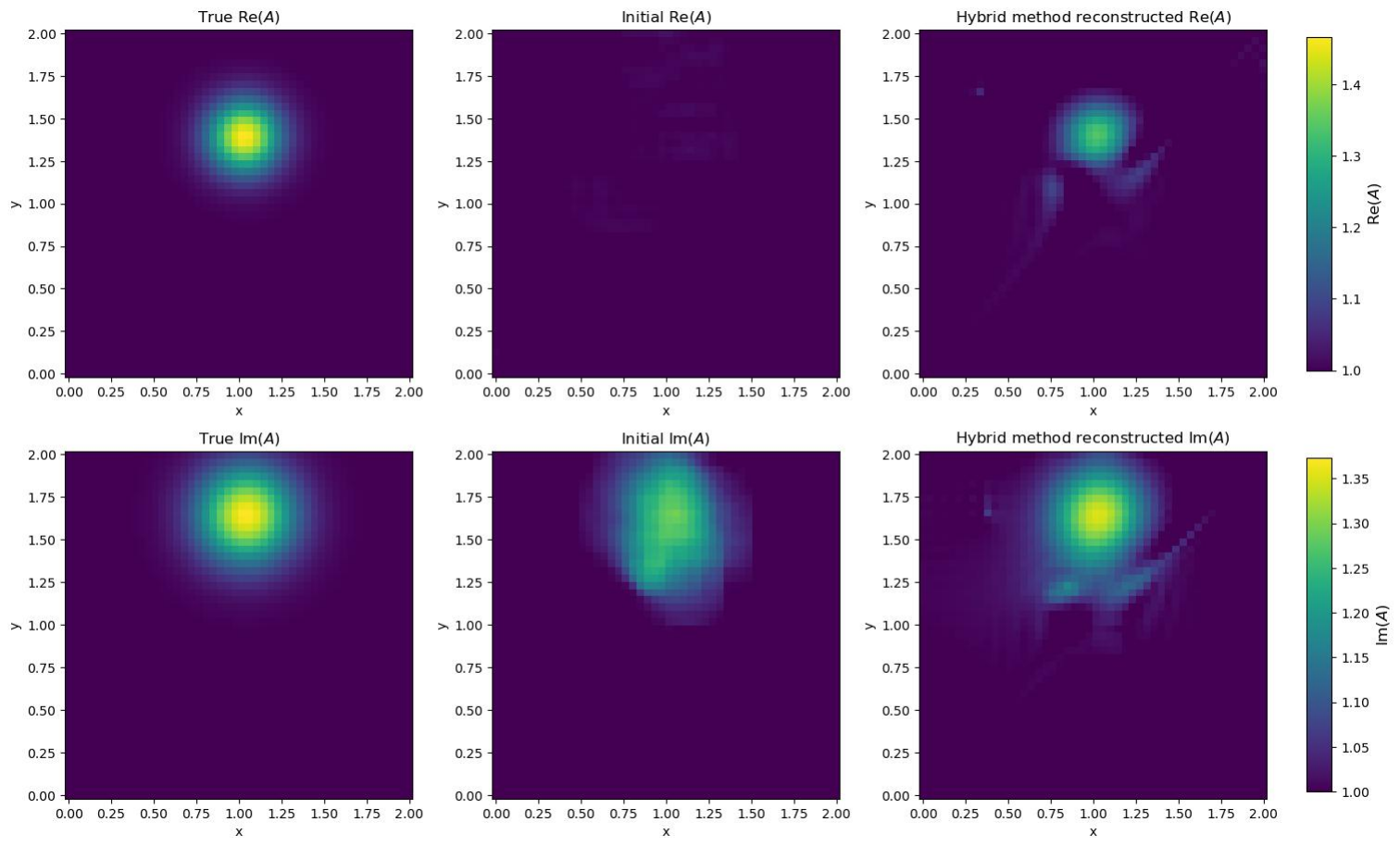}

\caption{End-to-end hybrid reconstruction for a representative
complex-valued Gaussian test sample from Experiment I.
The left column shows the true coefficient, the middle column shows
the CNN-generated initial estimate, and the right column shows the
coefficient reconstructed after PDE-constrained optimization.
The top and bottom rows correspond to the real and imaginary
components of \(A(x,y)\), respectively.}

\label{fig:gaussian_hybrid_comparison}
\end{figure}

To quantify the improvement obtained through the optimization stage, we
compare the reconstruction errors of the CNN-generated initial estimate with
those of the final hybrid reconstruction. The results are summarized in
Table~\ref{tab:ellipse_hybrid_errors} and Table~\ref{tab:gaussian_hybrid_errors}.

\begin{table}[H]
\centering
\caption{Reconstruction errors for the representative mollified-elliptical
disc-shaped test sample before and after the hybrid reconstruction framework.}
\label{tab:ellipse_hybrid_errors}

\begin{tabular}{lccc}
\hline
 & CNN initial estimate & Hybrid reconstruction & Error decrease (\%) \\
\hline

$\|\operatorname{Re}(A)-\operatorname{Re}(A_d)\|_{L^2}$
& $1.005910\times10^{-1}$
& $3.075439\times10^{-2}$
& $69.43$ \\

$\|\operatorname{Im}(A)-\operatorname{Im}(A_d)\|_{L^2}$
& $1.365835\times10^{-1}$
& $2.752220\times10^{-2}$
& $79.85$ \\

$\|A-A_d\|_{L^2}$
& $1.696279\times10^{-1}$
& $4.127111\times10^{-2}$
& $75.67$ \\

\hline
\end{tabular}
\end{table}

For the representative mollified-elliptical disc-shaped test sample, the
PDE-constrained optimization substantially improves the CNN initialization
for both components of the coefficient and for the full complex-valued
coefficient $A$.

\begin{table}[H]
\centering
\caption{Reconstruction errors for the representative Gaussian test sample
before and after the hybrid reconstruction framework.}
\label{tab:gaussian_hybrid_errors}

\begin{tabular}{lccc}
\hline
 & CNN initial estimate & Hybrid reconstruction & Error decrease (\%) \\
\hline

$\|\operatorname{Re}(A)-\operatorname{Re}(A_d)\|_{L^2}$
& $1.2582\times10^{-1}$ 
& $7.3615\times10^{-2}$ 
& $41.49$ \\

$\|\operatorname{Im}(A)-\operatorname{Im}(A_d)\|_{L^2}$
& $4.3595\times10^{-2}$ 
& $4.1262\times10^{-2}$ 
& $5.35$ \\

$\|A-A_d\|_{L^2}$
& $1.3316\times10^{-1}$ 
& $8.4390\times10^{-2}$ 
& $36.62$ \\

\hline
\end{tabular}
\end{table}

For both shape families, the CNN-generated estimate identifies the principal
locations and geometries of the real and imaginary perturbations and therefore
provides an informative initialization for the inverse problem. The subsequent
optimization stage refines this estimate by incorporating the governing PDE
and the internal displacement measurements. The results demonstrate that the
hybrid framework can be applied to both Gaussian and mollified-ellipse
coefficient families using a single CNN trained jointly on the two classes.

\paragraph{Robustness to noisy displacement data.}

We next investigate the robustness of the reconstruction framework when the
internal displacement measurements are contaminated by noise. The same
coefficient samples used in the noiseless experiment are considered, while
the corresponding displacement fields are perturbed before being supplied
to the CNN.

Let $u_d$ denote the noiseless internal displacement and let
$u_d^\delta$ denote the noisy measurement. The perturbation is scaled so
that
\begin{equation}
\frac{
\|u_d^\delta-u_d\|_{L^2(\Omega)}
}{
\|u_d\|_{L^2(\Omega)}
}
=
\delta.
\label{eq:relative_noise_exp1}
\end{equation}
We consider a relative noise level of
\[
\delta=0.01
\]
corresponding to $1\%$ 
noise, respectively.
The CNN used in the noiseless experiments is retained without retraining.
Thus, the noisy displacement fields are supplied directly to the network,
and the resulting coefficient predictions are used as initial guesses for
the subsequent PDE-constrained optimization. This experiment therefore
evaluates the sensitivity of both the CNN initialization and the complete
hybrid reconstruction framework to perturbations in the internal
displacement data.

To first assess the effect of measurement noise on the CNN prediction
independently of the subsequent optimization stage, we evaluate the
pretrained CNN over all $500$ test samples using displacement data
contaminated with $1\%$ relative noise. The $L^2(\Omega)$ reconstruction
errors for the real and imaginary components and for the complete
complex-valued coefficient are shown in
Figure~\ref{fig:cnn_noise1_error_histograms}.

\begin{figure}[H]
\centering

\begin{subfigure}{0.32\textwidth}
    \centering
    \includegraphics[width=\linewidth]
    {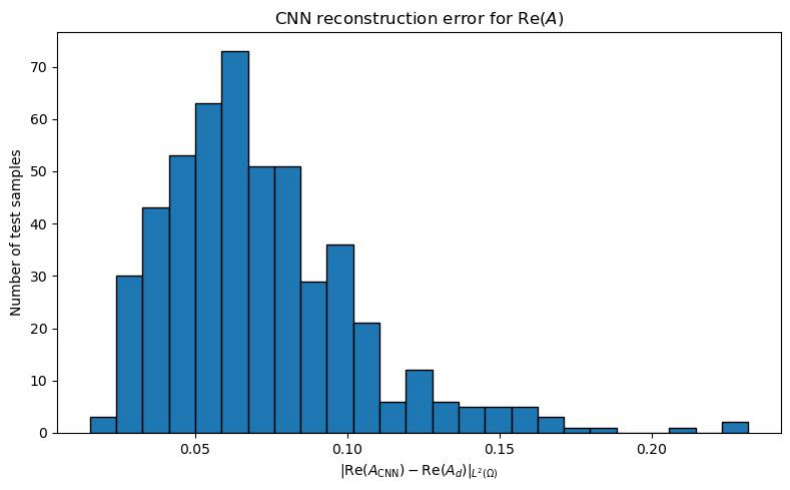}
    \caption{$L^2(\Omega)$ reconstruction errors for
    $\operatorname{Re}(A)$.}
    \label{fig:cnn_noise1_real}
\end{subfigure}
\hfill
\begin{subfigure}{0.32\textwidth}
    \centering
    \includegraphics[width=\linewidth]
    {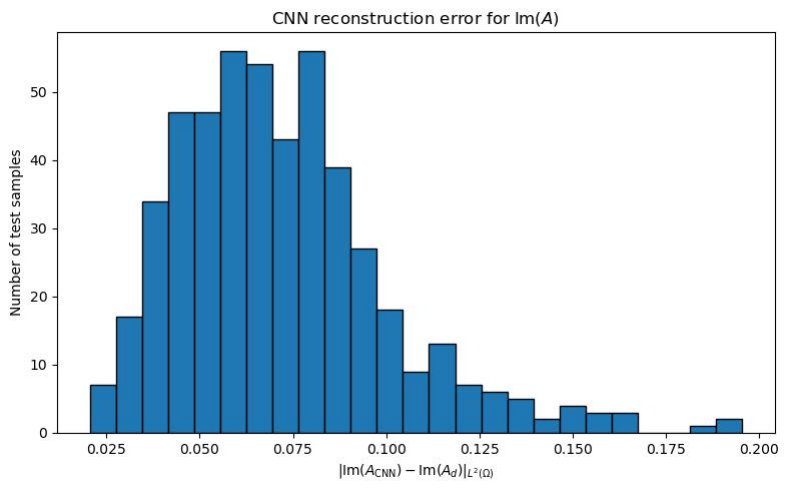}
    \caption{$L^2(\Omega)$ reconstruction errors for
    $\operatorname{Im}(A)$.}
    \label{fig:cnn_noise1_imag}
\end{subfigure}
\hfill
\begin{subfigure}{0.32\textwidth}
    \centering
    \includegraphics[width=\linewidth]
    {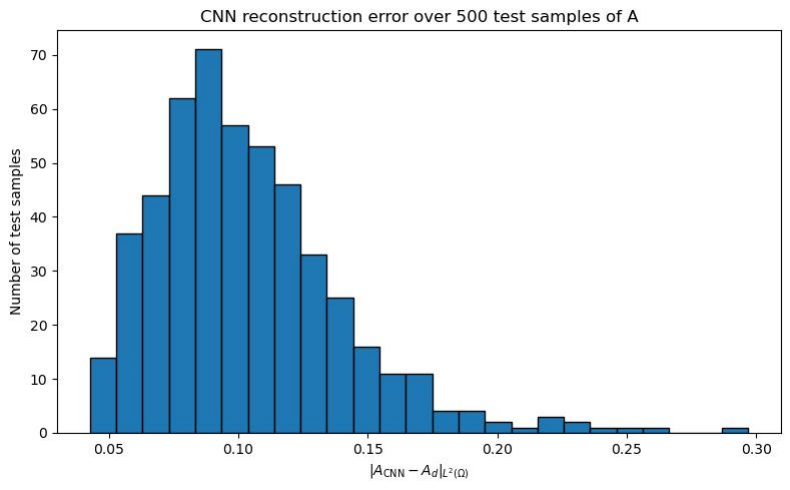}
    \caption{$L^2(\Omega)$ reconstruction errors for the
    complex-valued coefficient $A$.}
    \label{fig:cnn_noise1_complex}
\end{subfigure}

\caption{Distribution of the CNN reconstruction errors over all $500$
test samples using internal displacement data contaminated with $1\%$
relative noise. The CNN is trained using noiseless displacement data
and is applied directly to the noisy measurements without retraining.}
\label{fig:cnn_noise1_error_histograms}

\end{figure}

\begin{table}[H]
\centering
\caption{Summary statistics of the CNN reconstruction errors for the
$1\%$ noisy test data.}
\label{tab:cnn_noise1_test_errors}

\begin{tabular}{lccccc}
\hline
 & Mean & Std. & Median & Min. & Max. \\
\hline

$\operatorname{Re}(A)$
& $7.1454\times10^{-2}$
& $3.2029\times10^{-2}$
& $6.5777\times10^{-2}$
& $1.5458\times10^{-2}$
& $2.3168\times10^{-1}$ \\

$\operatorname{Im}(A)$
& $7.1907\times10^{-2}$
& $2.8521\times10^{-2}$
& $6.7771\times10^{-2}$
& $2.0490\times10^{-2}$
& $1.9553\times10^{-1}$ \\

$A$
& $1.0377\times10^{-1}$
& $3.6707\times10^{-2}$
& $9.7522\times10^{-2}$
& $4.2700\times10^{-2}$
& $2.9684\times10^{-1}$ \\

\hline
\end{tabular}
\end{table}

The error distributions remain concentrated at relatively small values,
with right-skewed tails corresponding to a smaller number of test samples
with larger reconstruction errors. Overall, the results indicate that the
CNN retains reasonable reconstruction accuracy under $1\%$ measurement
noise despite being trained exclusively on noiseless data.

To further examine the performance of the complete hybrid framework in the
presence of measurement noise, we consider a representative Gaussian
coefficient from the test dataset. The displacement measurement is
contaminated with $1\%$ relative noise and supplied to the pretrained CNN.
The resulting CNN prediction is then used as the initial guess for the
PDE-constrained optimization. Figure~\ref{fig:gaussian_noise1_reconstruction}
compares the true coefficient, the CNN initialization, and the final hybrid
reconstruction.
The optimization
and line-search parameters used for the $1\%$ noise experiment are
summarized in Table~\ref{tab:noise1_opt_parameters}.

\begin{table}[H]
\centering
\caption{Optimization parameters used for the reconstruction of the Gaussian
and mollified-ellipse inclusions with $1\%$ relative noise.}
\label{tab:noise1_opt_parameters}

\begin{tabular}{lcc}
\hline
\textbf{Parameter}
& \textbf{Gaussian}
& \textbf{mollified elliptical disc} \\
\hline

Data-fidelity parameter $\alpha$
& $20$
& $20$ \\

Regularization parameter $\beta$
& $1\times10^{-5}$
& $1\times10^{-6}$ \\

Maximum number of iterations $k_{\max}$
& $200$
& $150$ \\

Initial line-search step size
& $1$
& $1$ \\

Armijo parameter
& $10^{-4}$
& $10^{-4}$ \\

Step-size reduction factor
& $0.8$
& $0.8$ \\

\hline
\end{tabular}
\end{table}

\begin{figure}[H]
    \centering
    \includegraphics[width=0.7\textwidth]
    {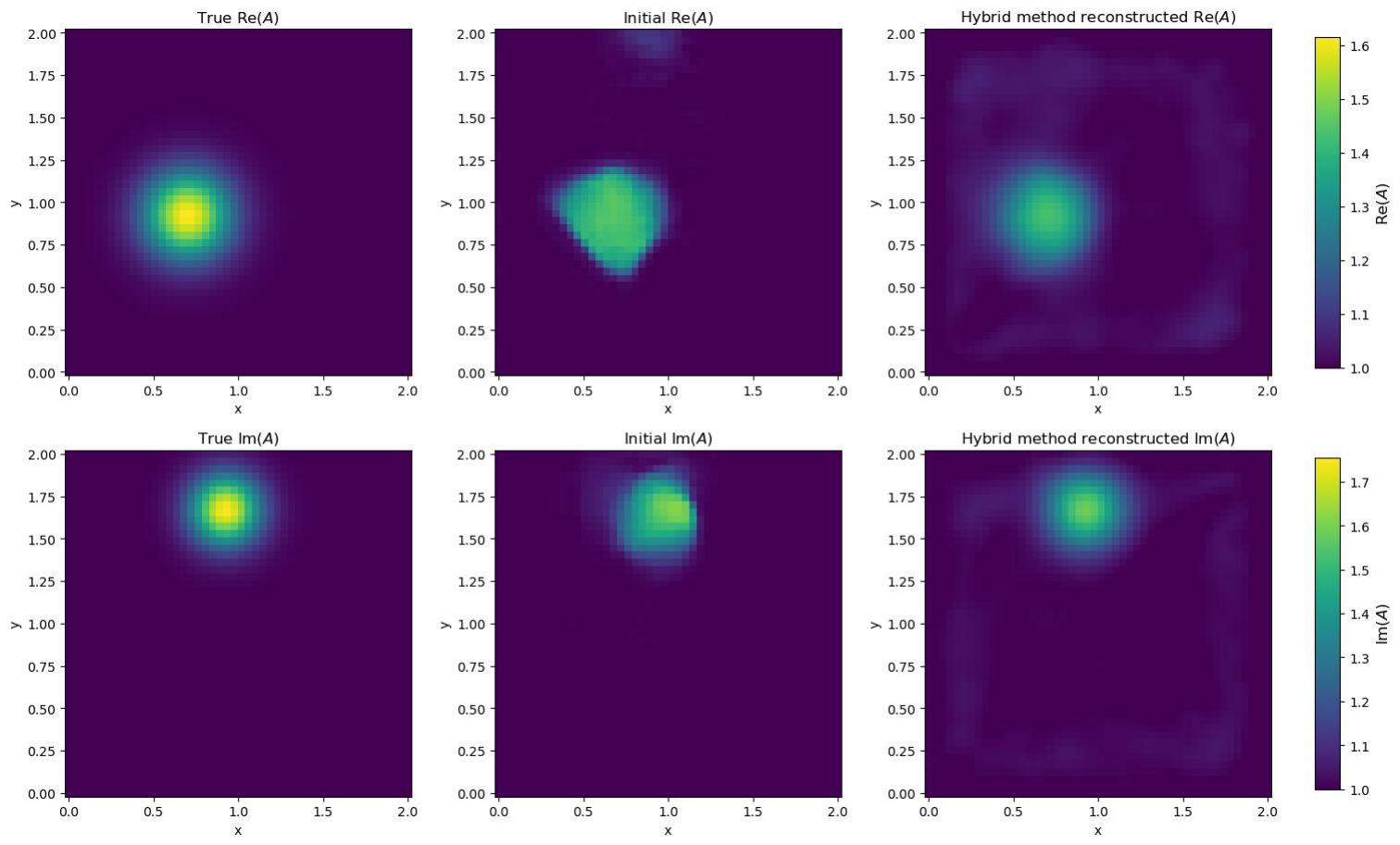}
    \caption{Reconstruction of a representative Gaussian coefficient from
    internal displacement data contaminated with $1\%$ relative noise.
    The left, middle, and right columns show the true coefficient, the
    CNN-generated initial estimate, and the final hybrid reconstruction,
    respectively. The top and bottom rows correspond to the real and
    imaginary components of $A(x,y)$.}
    \label{fig:gaussian_noise1_reconstruction}
\end{figure}

The corresponding coefficient reconstruction errors are reported in
Table~\ref{tab:gaussian_noise1_errors}.

\begin{table}[H]
\centering
\caption{Reconstruction errors for the representative Gaussian coefficient
using displacement data contaminated with $1\%$ relative noise.}
\label{tab:gaussian_noise1_errors}

\begin{tabular}{lccc}
\hline
\textbf{Reconstruction error}
& \textbf{CNN initialization}
& \textbf{Hybrid reconstruction}
& \textbf{Error decrease (\%)} \\
\hline

$\|\operatorname{Re}(A)-\operatorname{Re}(A_d)\|_{L^2(\Omega)}$
& $6.291202\times10^{-2}$
& $6.229779\times10^{-2}$
& $0.98\%$ \\

$\|\operatorname{Im}(A)-\operatorname{Im}(A_d)\|_{L^2(\Omega)}$
& $5.586334\times10^{-2}$
& $5.436543\times10^{-2}$
& $2.68\%$ \\

$\|A-A_d\|_{L^2(\Omega)}$
& $8.413462\times10^{-2}$
& $8.268383\times10^{-2}$
& $1.72\%$ \\

\hline
\end{tabular}
\end{table}

For the mollified-ellipse inclusion, the CNN prediction is again used as
the initial iterate for the PDE-constrained optimization. The optimization
parameters used for this reconstruction are summarized in
Table~\ref{tab:noise1_opt_parameters}.

Figure~\ref{fig:ellipse_noise1_reconstruction} compares the true
mollified-ellipse coefficient, the CNN-generated initial estimate, and
the final hybrid reconstruction.

\begin{figure}[H]
    \centering
    \includegraphics[width=0.7\textwidth]
    {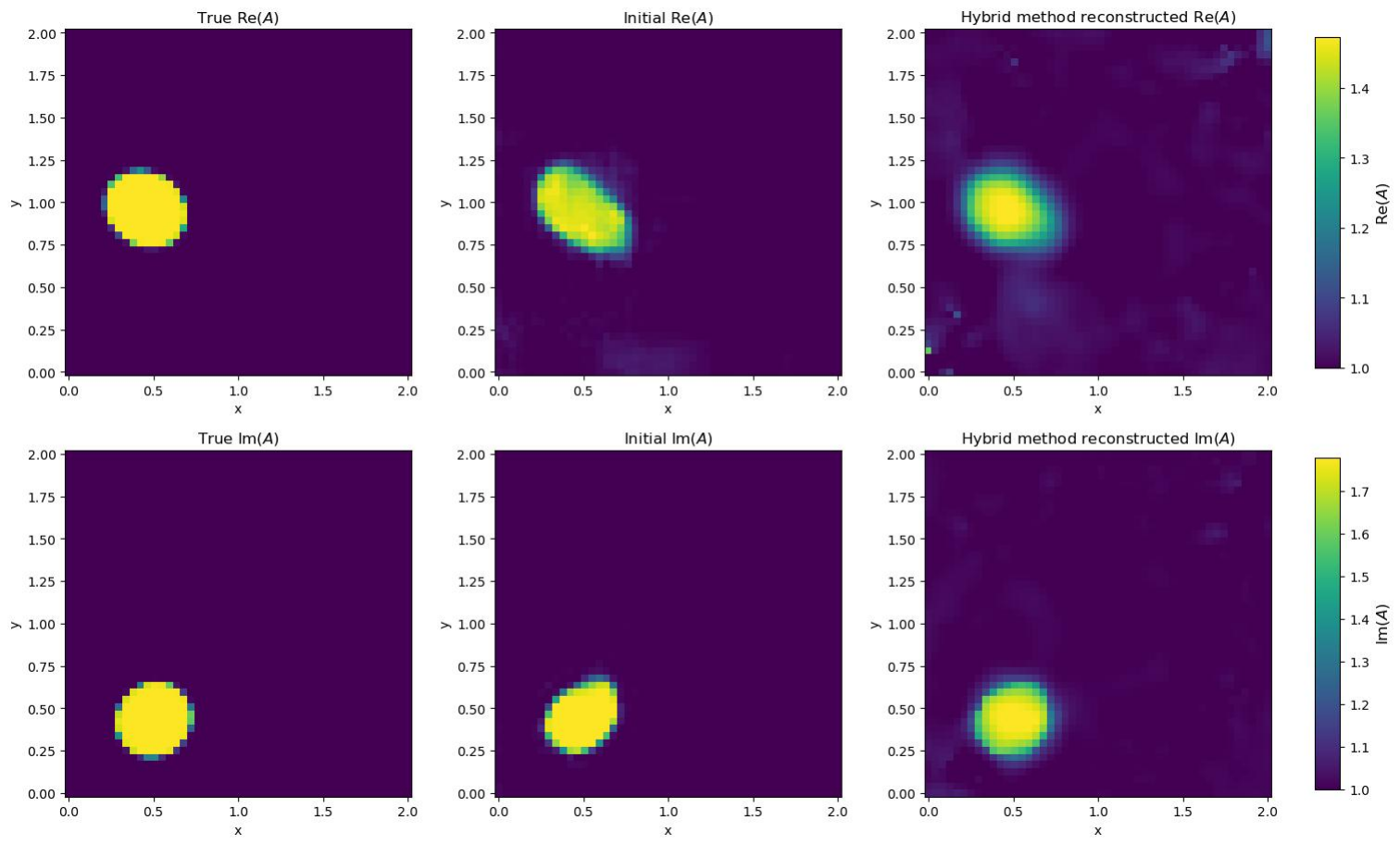}
    \caption{Reconstruction of a representative mollified-ellipse
    coefficient from internal displacement data contaminated with $1\%$
    relative noise. The left, middle, and right columns show the true
    coefficient, the CNN-generated initial estimate, and the final hybrid
    reconstruction, respectively. The top and bottom rows correspond to
    the real and imaginary components of $A(x,y)$.}
    \label{fig:ellipse_noise1_reconstruction}
\end{figure}

The corresponding coefficient reconstruction errors are reported in
Table~\ref{tab:ellipse_noise1_errors}.

\begin{table}[H]
\centering
\caption{Reconstruction errors for the representative mollified-ellipstical disc
coefficient using displacement data contaminated with $1\%$ relative noise.}
\label{tab:ellipse_noise1_errors}

\begin{tabular}{lccc}
\hline
\textbf{Reconstruction error}
& \textbf{CNN initialization}
& \textbf{Hybrid reconstruction}
& \textbf{Error decrease} \\
\hline

$\|\operatorname{Re}(A)-\operatorname{Re}(A_d)\|_{L^2(\Omega)}$
& $7.901710\times10^{-2}$
& $7.279785\times10^{-2}$
& $7.87\%$ \\

$\|\operatorname{Im}(A)-\operatorname{Im}(A_d)\|_{L^2(\Omega)}$
& $1.040895\times10^{-1}$
& $6.924087\times10^{-2}$
& $33.48\%$ \\

$\|A-A_d\|_{L^2(\Omega)}$
& $1.306841\times10^{-1}$
& $1.004680\times10^{-1}$
& $23.12\%$ \\

\hline
\end{tabular}
\end{table}

The PDE-constrained optimization further refines the localized coefficient
features obtained from the CNN initialization. The principal Gaussian and mollified-elliptical
inclusions remain well recovered in both the real and imaginary components
of the coefficient. However, the optimized solution exhibits oscillatory
artifacts away from the inclusions. These artifacts illustrate the increased
sensitivity of the PDE-constrained reconstruction to perturbations in the
measured displacement data. In the presence of noisy
data, continued reduction of the displacement misfit may also introduce
spatial variations associated with the measurement perturbations. Thus,
although the dominant coefficient features remain identifiable, the noisy
experiments indicate that the choice of regularization strength and stopping
criterion becomes increasingly important. These observations motivate a
more careful treatment of regularization and early stopping for noisy
measurements.

\subsubsection{Experiment II: reconstruction of combined Gaussian and mollified-ellipse structures}

We next consider a more challenging test problem in which the coefficient
fields contain both shape families simultaneously. A second test dataset of
$500$ samples is generated, where every sample contains one Gaussian and one
mollified-ellipse perturbation in both the real and imaginary components of
the coefficient.

More precisely, we define
\begin{equation}
B_1(x,y)
=
G_r(x,y)+E_r(x,y),
\end{equation}
and
\begin{equation}
B_2(x,y)
=
G_i(x,y)+E_i(x,y),
\end{equation}
where $G_r$ and $G_i$ denote the Gaussian perturbations in the real and
imaginary components, respectively, while $E_r$ and $E_i$ denote the
corresponding mollified-ellipse perturbations. The resulting complex-valued
coefficient is therefore given by
\begin{equation}
A(x,y)
=
\left[1+G_r(x,y)+E_r(x,y)\right]
+
i\left[1+G_i(x,y)+E_i(x,y)\right].
\label{eq:combined_test_coeff}
\end{equation}

The parameters defining the Gaussian and mollified-ellipse structures are
generated independently for the real and imaginary components. For the
Gaussian perturbations, the amplitudes are sampled according to
\[
a_r^G,a_i^G\sim\mathcal{U}(0.2,0.8),
\]
the Gaussian spreads satisfy
\[
\sigma_r,\sigma_i\sim\mathcal{U}(0.1,0.3),
\]
and the center coordinates are sampled from
\[
c_{x,r}^G,c_{y,r}^G,c_{x,i}^G,c_{y,i}^G
\sim\mathcal{U}(0.3,1.7).
\]

For the mollified-ellipse perturbations, the amplitudes satisfy
\[
a_r^E,a_i^E\sim\mathcal{U}(0.2,0.8),
\]
while the semi-axis lengths are sampled according to
\[
r_{x,r},r_{x,i}\sim\mathcal{U}(0.2,0.3),
\qquad
r_{y,r},r_{y,i}\sim\mathcal{U}(0.1,0.3).
\]
The rotation angles and smoothing parameters are independently sampled from
\[
\theta_r,\theta_i\sim\mathcal{U}(0,\pi),
\qquad
\varepsilon_r,\varepsilon_i\sim\mathcal{U}(0.03,0.08).
\]
The centers of the mollified elliptical discs are restricted according to their
semi-axis lengths and smoothing parameters so that the inclusions remain
inside the computational domain.

Thus, Experiment II differs substantially from Experiment I. The CNN was
trained using Gaussian and mollified-ellipse coefficient fields as separate
samples. In particular, during training the network never encountered a
coefficient field in which a Gaussian and a mollified elliptical disc were present
simultaneously within the same spatial domain. Experiment I preserved the
structure of the training data, since each test coefficient contains only
one of the two shape families. In Experiment II, however, both shape families
occur simultaneously in each component of the coefficient, so that a single
coefficient field contains both a Gaussian and a mollified-ellipse
perturbation. Consequently, the CNN has not been trained directly on
coefficient fields of the form~\eqref{eq:combined_test_coeff}. This experiment
therefore provides a more challenging generalization test and investigates
whether the spatial features learned separately for the two shape families
can be transferred to coefficient fields containing their superposition.

As in Experiment I, the internal displacement field corresponding to each
coefficient is obtained by solving the forward PDE using the same Dirichlet
boundary excitation used during training. The test data considered in
this experiment are noiseless.

Table~\ref{tab:cnn_combined_test_errors} summarizes the reconstruction
errors over all $500$ combined-shape test samples.

\begin{table}[H]
\centering
\caption{CNN reconstruction errors over the $500$ combined Gaussian and
mollified-ellipse test samples in Experiment II.}
\label{tab:cnn_combined_test_errors}

\begin{tabular}{lccccc}
\hline
 & Mean & Std. & Median & Min. & Max. \\
\hline

$\operatorname{Re}(A)$
& $1.279967\times10^{-1}$
& $4.694296\times10^{-2}$
& $1.202852\times10^{-1}$
& $4.829479\times10^{-2}$
& $3.416799\times10^{-1}$ \\

$\operatorname{Im}(A)$
& $1.290662\times10^{-1}$
& $4.606712\times10^{-2}$
& $1.234582\times10^{-1}$
& $4.722147\times10^{-2}$
& $3.733799\times10^{-1}$ \\

$A$
& $1.852892\times10^{-1}$
& $5.509064\times10^{-2}$
& $1.760132\times10^{-1}$
& $8.382436\times10^{-2}$
& $4.479576\times10^{-1}$ \\

\hline
\end{tabular}
\end{table}

The distributions of these errors over the complete test dataset are shown
in Figure~\ref{fig:cnn_combined_error_histograms}. The error
distributions therefore provide a quantitative assessment of the ability
of the CNN to generalize from individual Gaussian and mollified-ellipse
structures to coefficients containing their superposition.

\begin{figure}[H]
\centering

\begin{subfigure}{0.32\textwidth}
    \centering
    \includegraphics[width=\linewidth]
    {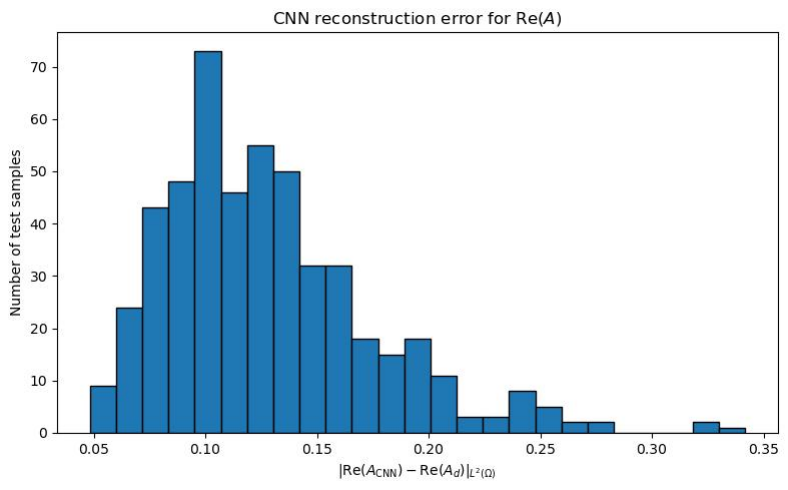}
    \caption{$L^2(\Omega)$ reconstruction errors for
    $\operatorname{Re}(A)$.}
\end{subfigure}
\hfill
\begin{subfigure}{0.32\textwidth}
    \centering
    \includegraphics[width=\linewidth]
    {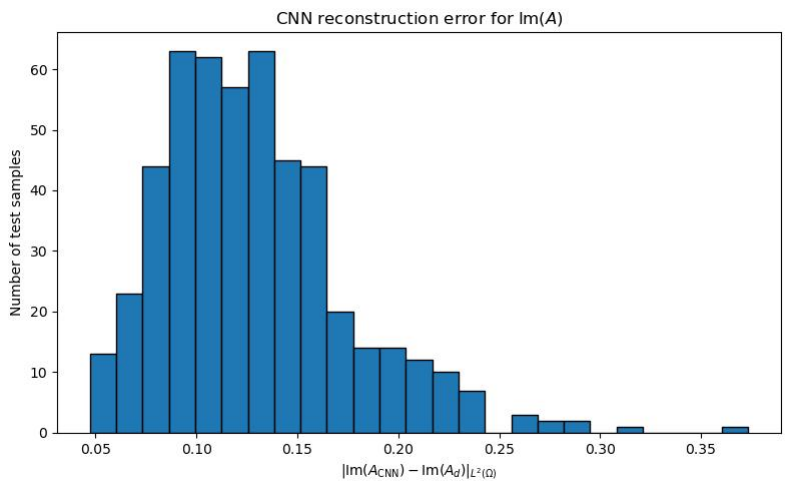}
    \caption{$L^2(\Omega)$ reconstruction errors for
    $\operatorname{Im}(A)$.}
\end{subfigure}
\hfill
\begin{subfigure}{0.32\textwidth}
    \centering
    \includegraphics[width=\linewidth]
    {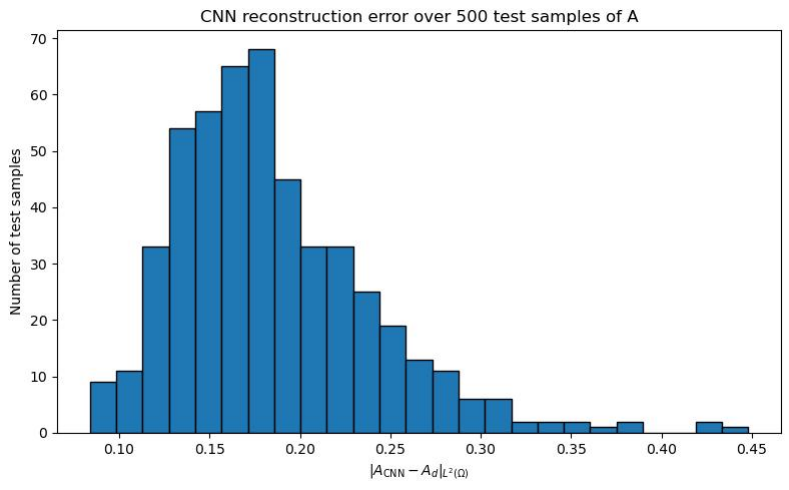}
    \caption{$L^2(\Omega)$ reconstruction errors for the
    complex-valued coefficient $A$.}
\end{subfigure}

\caption{Distribution of the CNN reconstruction errors over all $500$
test samples in Experiment II. Each real and imaginary component of the
coefficient contains both a Gaussian and a mollified-ellipse perturbation.}
\label{fig:cnn_combined_error_histograms}

\end{figure}

Having evaluated the CNN over the complete test dataset, we next examine
representative combined-shape reconstructions in more detail. We first
consider a sample in which the Gaussian and mollified-ellipse structures
are spatially separated. We subsequently consider a more challenging
configuration in which the two structures overlap. This example
illustrates the spatial structures recovered by the CNN when both inclusion
types are present simultaneously and the quality of the resulting initial
coefficient field supplied to the subsequent PDE-constrained optimization
stage.

For a representative test coefficient containing spatially separated
Gaussian and mollified-ellipse perturbations, the corresponding internal
displacement field is shown in
Figure~\ref{fig:combined_displacement_input}. The real and imaginary
components of this complex-valued displacement field constitute the two
input channels supplied to the CNN.

\begin{figure}[H]
\centering
\includegraphics[width=0.5\linewidth]
{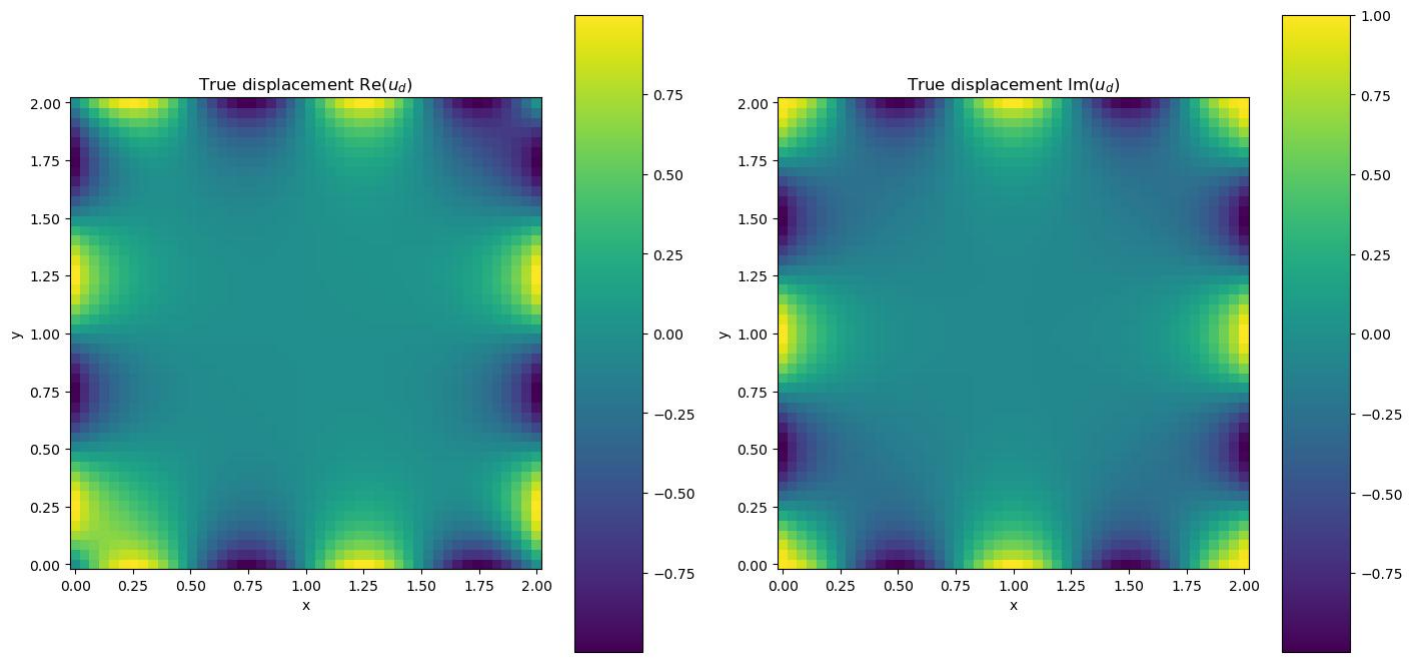}

\caption{Real and imaginary components of the internal displacement field
$u_d$ corresponding to a representative test coefficient containing
combined Gaussian and mollified-ellipse perturbations. These displacement
data form the two input channels to the CNN.}

\label{fig:combined_displacement_input}
\end{figure}

Using the displacement field in
Figure~\ref{fig:combined_displacement_input}, the CNN produces the
coefficient estimate shown in Figure~\ref{fig:CNN_combined_sample}.
Although the network was trained only on samples containing individual
Gaussian or mollified-ellipse perturbations, the prediction identifies the
principal locations and approximate geometries of the multiple structures
present in the coefficient. Some discrepancies remain in the amplitudes,
shapes, and transition regions of the reconstructed inclusions.

\begin{figure}[H]
\centering

\begin{subfigure}{0.5\textwidth}
    \centering
    \includegraphics[width=\linewidth]
    {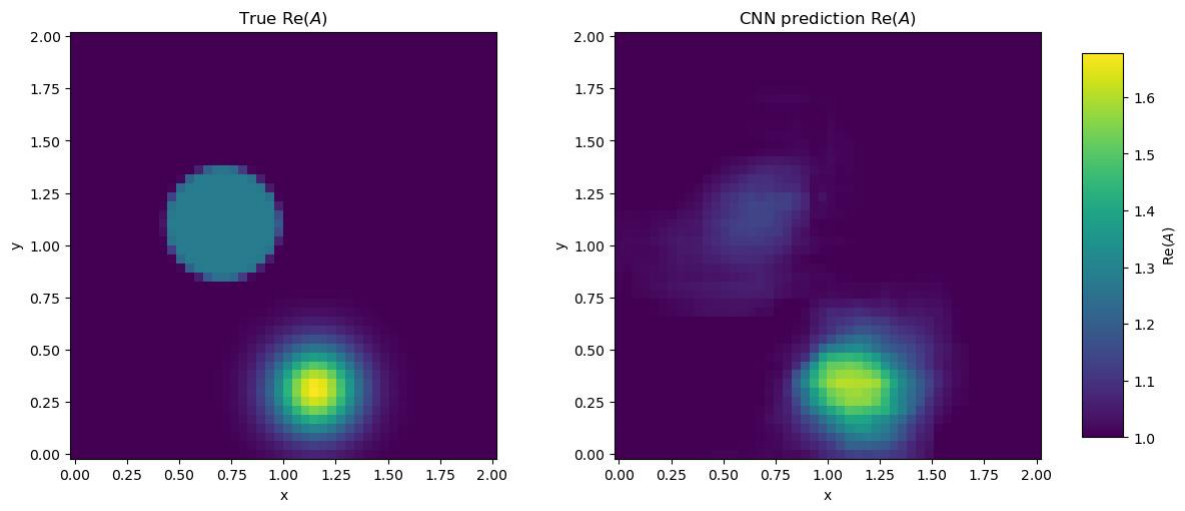}
    \caption{Real component of the coefficient.}
    \label{fig:CNN_combined_real}
\end{subfigure}

\begin{subfigure}{0.5\textwidth}
    \centering
    \includegraphics[width=\linewidth]
    {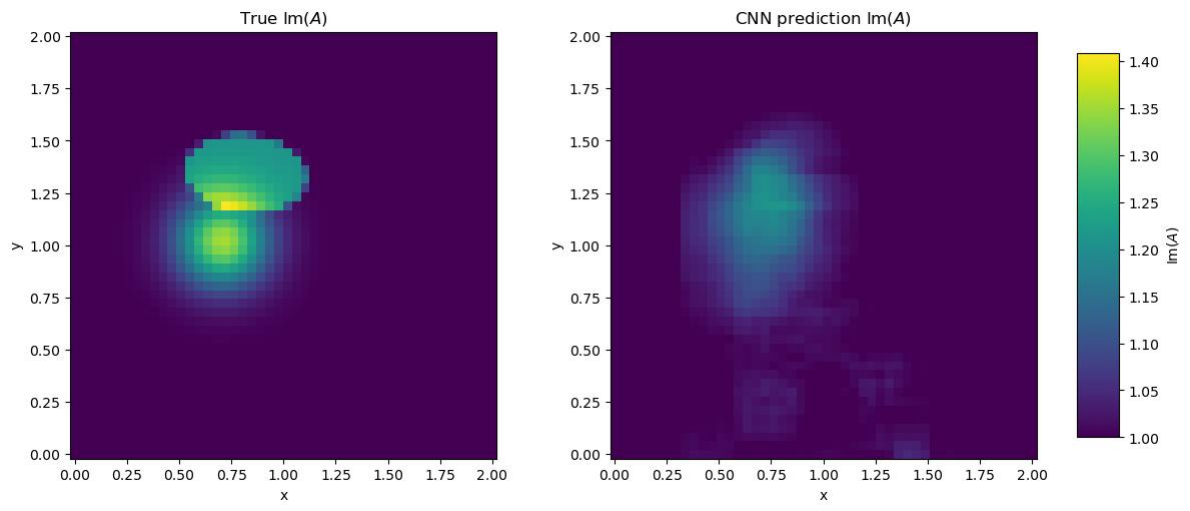}
    \caption{Imaginary component of the coefficient.}
    \label{fig:CNN_combined_imag}
\end{subfigure}

\caption{Representative CNN prediction for a combined Gaussian and
mollified-ellipse test sample from Experiment II. The top row compares
the true and CNN-predicted real components of $A(x,y)$, while the bottom
row compares the corresponding imaginary components. Both components
contain one Gaussian and one independently generated mollified elliptical disc.}

\label{fig:CNN_combined_sample}
\end{figure}

The CNN prediction in Figure~\ref{fig:CNN_combined_sample} provides an
informative approximation of the unknown coefficient over the complete
spatial domain despite the increased structural complexity of the test
sample. In particular, the network is able to identify multiple localized
features even though such combined configurations were not explicitly
included in the training dataset. The CNN-generated coefficient is
therefore used as the initial guess for the PDE-constrained optimization.

\paragraph{Optimization setup.}

For the optimization stage of the hybrid framework, the CNN-predicted
coefficient is used as the initial iterate for the PDE-constrained
optimization. The reconstruction is
then refined by minimizing the regularized PDE-constrained objective
functional introduced previously. In all reconstructions considered in Experiment II, we use
the NCG method with an Armijo backtracking line search and the fixed
optimization parameters listed in Table~\ref{tab:experiment2_optimization_parameters}.

\begin{table}[H]
\centering
\caption{Optimization parameters used throughout Experiment II.}
\label{tab:experiment2_optimization_parameters}

\begin{tabular}{lc}
\hline
Parameter & Value \\
\hline
Data-fidelity weight $\alpha$ & 20 \\
Regularization parameter $\beta$ & $1\times10^{-7}$ \\
Maximum NCG iterations $k_{\max}$ & 2500 \\
Initial line-search step size & 1 \\
Armijo parameter & $1\times10^{-4}$ \\
Backtracking reduction factor & 0.8 \\
\hline
\end{tabular}
\end{table}

To directly illustrate the contribution of the optimization stage, we
compare the true coefficient, the CNN-generated initial estimate, and the
coefficient obtained after PDE-constrained optimization. The corresponding
comparison is shown in Figure~\ref{fig:combined_hybrid_comparison}.

\begin{figure}[H]
\centering
\includegraphics[width=0.7\linewidth]
{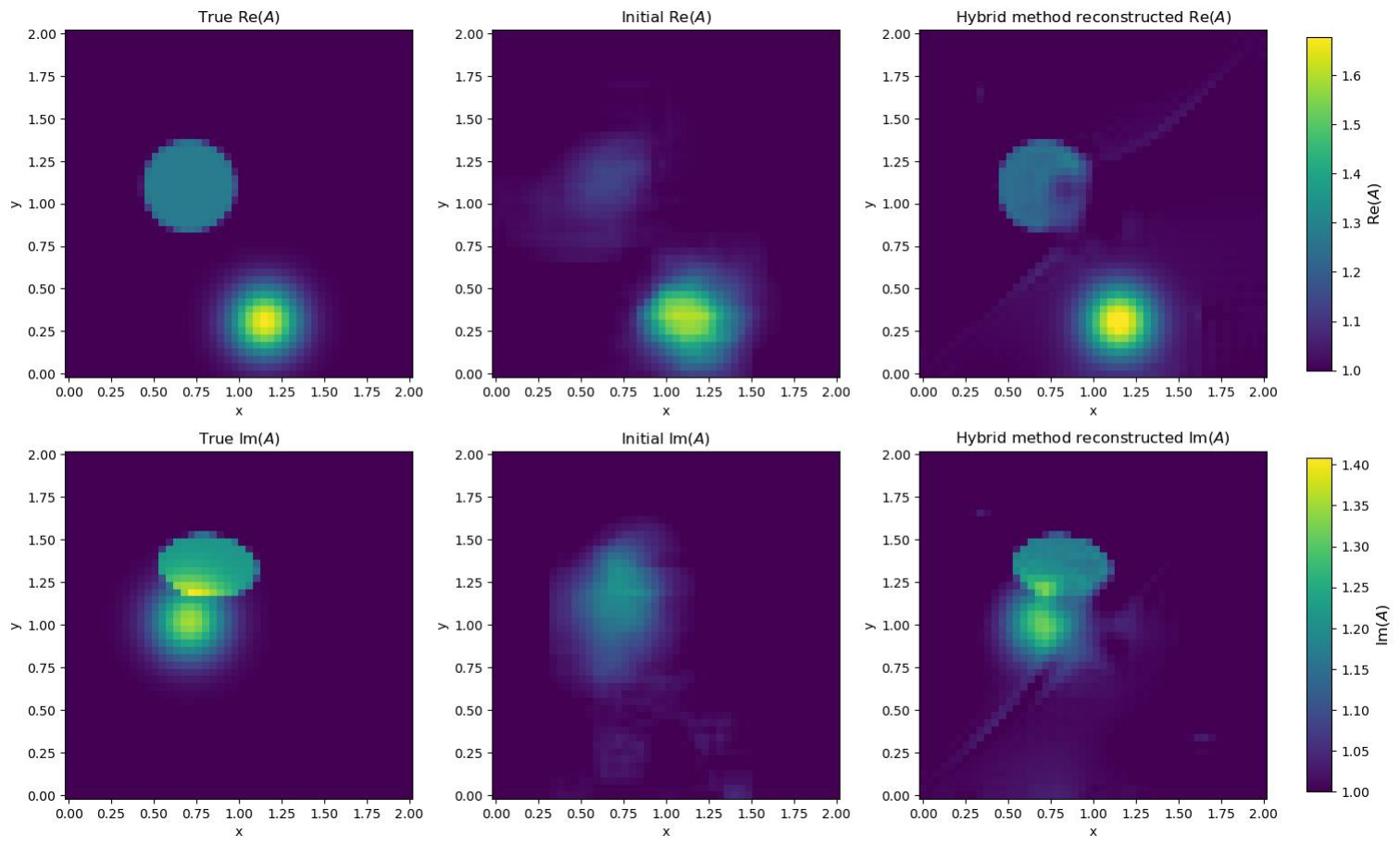}

\caption{End-to-end hybrid reconstruction for a representative combined
Gaussian and mollified-ellipse test sample from Experiment II. The left
column shows the true coefficient, the middle column shows the CNN-generated
initial estimate, and the right column shows the coefficient reconstructed
after PDE-constrained optimization. The top and bottom rows correspond to
the real and imaginary components of $A(x,y)$, respectively.}

\label{fig:combined_hybrid_comparison}
\end{figure}

To quantify the improvement obtained through the optimization stage, we
compare the reconstruction errors of the CNN-generated initial estimate
with those of the final hybrid reconstruction. The results are summarized
in Table~\ref{tab:combined_hybrid_errors}.

\begin{table}[H]
\centering
\caption{Reconstruction errors for the representative combined Gaussian
and mollified-ellipse test sample before and after the hybrid
reconstruction framework.}
\label{tab:combined_hybrid_errors}

\begin{tabular}{lccc}
\hline
 & CNN initial estimate & Hybrid reconstruction & Error decrease (\%) \\
\hline

$\|\operatorname{Re}(A)-\operatorname{Re}(A_d)\|_{L^2}$
& $1.097275\times10^{-1}$
& $5.953433\times10^{-2}$
& $45.74$ \\

$\|\operatorname{Im}(A)-\operatorname{Im}(A_d)\|_{L^2}$
& $7.435254\times10^{-2}$
& $4.500365\times10^{-2}$
& $39.47$ \\

$\|A-A_d\|_{L^2}$
& $1.325459\times10^{-1}$
& $7.463019\times10^{-2}$
& $43.70$ \\

\hline
\end{tabular}
\end{table}

Although some localized artifacts appear in the optimized reconstruction,
the PDE-constrained refinement substantially improves the overall
coefficient recovery.

We further examine a more challenging configuration in which the Gaussian
and mollified-ellipse perturbations overlap spatially. In contrast to the
previous representative sample, the two structures are no longer clearly
separated within the domain. Instead, their superposition produces a
coefficient variation in which the spatial features of the Gaussian and
mollified elliptical disc interact within the same region.
The overlapping case consequently provides a stronger test of whether the
features learned from the individual shape families can be transferred to
more complicated coefficient configurations.

\begin{figure}[H]
\centering
\includegraphics[width=0.7\linewidth]
{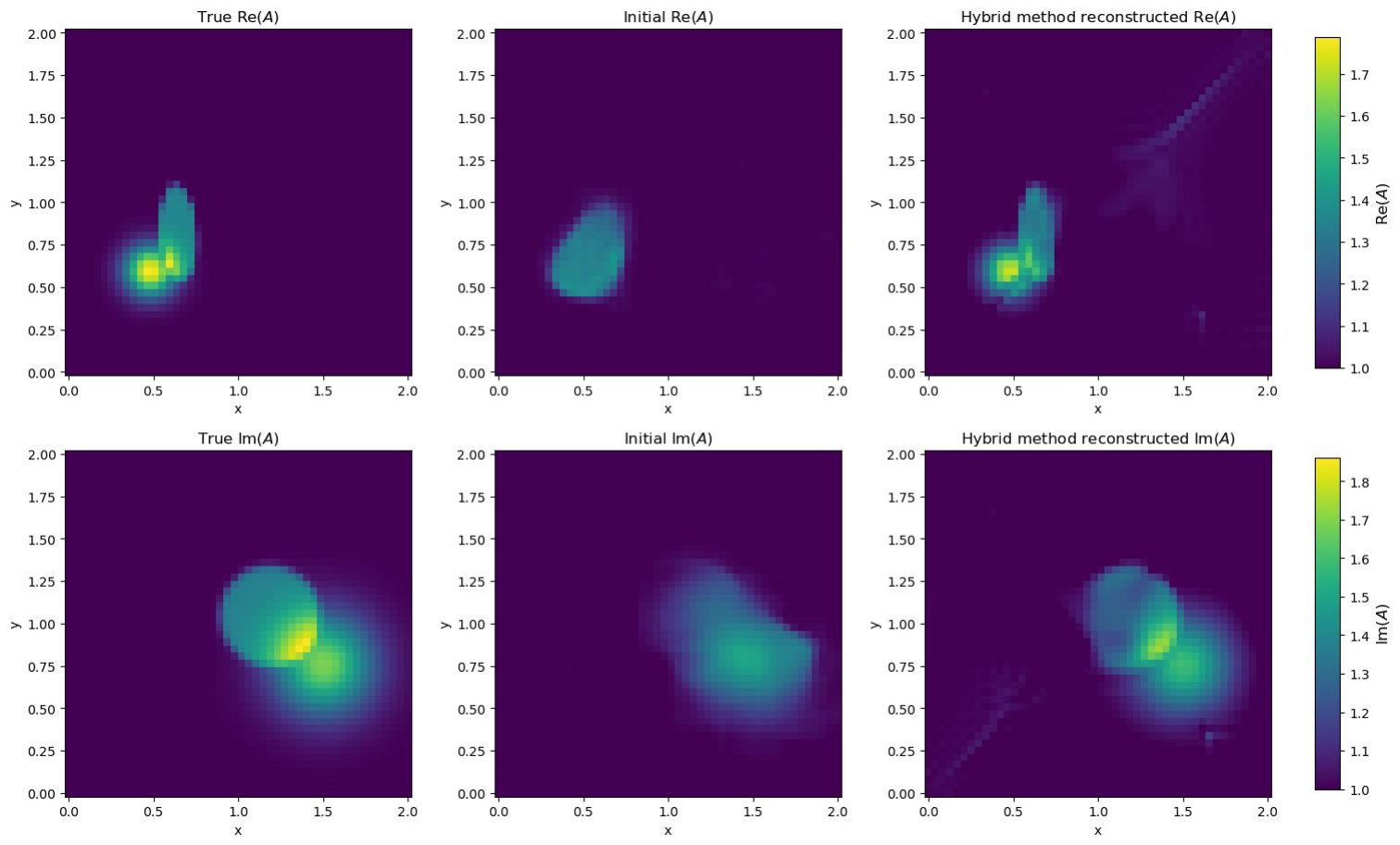}

\caption{End-to-end hybrid reconstruction for a representative test sample
with spatially overlapping Gaussian and mollified-ellipse perturbations.
The left column shows the true coefficient, the middle column shows the
CNN-generated initial estimate, and the right column shows the coefficient
obtained after PDE-constrained optimization. The top and bottom rows
correspond to the real and imaginary components of $A(x,y)$, respectively.
This example provides a more challenging generalization test because the
two inclusion types occupy overlapping spatial regions, while the CNN was
trained only on the individual shape families.}

\label{fig:combined_overlap_hybrid_comparison}
\end{figure}

\begin{table}[H]
\centering
\caption{Reconstruction errors for the representative test sample with
overlapping Gaussian and mollified-ellipse perturbations.}
\label{tab:combined_overlap_hybrid_errors}

\begin{tabular}{lccc}
\hline
 & CNN initial estimate & Hybrid reconstruction & Error decrease (\%) \\
\hline

$\|\operatorname{Re}(A)-\operatorname{Re}(A_d)\|_{L^2}$
& $8.751159\times10^{-2}$
& $5.142576\times10^{-2}$
& $41.24$ \\

$\|\operatorname{Im}(A)-\operatorname{Im}(A_d)\|_{L^2}$
& $1.562200\times10^{-1}$
& $1.140009\times10^{-1}$
& $27.03$ \\

$\|A-A_d\|_{L^2}$
& $1.790614\times10^{-1}$
& $1.250632\times10^{-1}$
& $30.16$ \\

\hline
\end{tabular}
\end{table}

For the representative overlapping configuration, the PDE-constrained
optimization improves the CNN-generated initial reconstruction for both
components and for the full complex-valued coefficient $A$.

The results of Experiment II investigate the generalization capability of
the hybrid reconstruction framework beyond the coefficient configurations
used for CNN training. Although the network was trained on Gaussian and
mollified-ellipse structures separately, its prediction provides an
informative initial estimate when both structures occur simultaneously.
The subsequent PDE-constrained optimization then refines this data-driven
estimate by enforcing consistency with the governing forward model and the
measured internal displacement field. Thus, the CNN and PDE-constrained
optimization play complementary roles: the CNN supplies a spatially
informed initialization, while the optimization incorporates the physics
of the inverse problem to improve the final reconstruction.

\paragraph{Robustness to noisy displacement data.}

We further investigate the robustness of the reconstruction framework for
the combined-shape configurations of Experiment II when the internal
displacement measurements are contaminated by noise. The same coefficient
samples used in the noiseless experiment are retained, while the
corresponding displacement fields are perturbed before being supplied to
the CNN.

As in Experiment I, let $u_d$ denote the noiseless internal displacement
and let $u_d^\delta$ denote the corresponding noisy measurement. The
perturbation is scaled according to \eqref{eq:relative_noise_exp1}. In this
experiment, we consider the relative noise level
\[
\delta = 0.01,
\]
corresponding to $1\%$ noise.

The CNN is retained without retraining and is applied directly to the noisy
displacement fields. The network was trained only on coefficient
fields containing either a Gaussian or a mollified-ellipse inclusion,
whereas each test coefficient in Experiment II contains both inclusion
types simultaneously. Consequently, the present experiment evaluates the
CNN under the combined effects of measurement noise and an unseen
coefficient configuration. The resulting CNN predictions are subsequently
used as initial guesses for the PDE-constrained optimization.

\begin{figure}[H]
\centering

\begin{subfigure}{0.32\textwidth}
    \centering
    \includegraphics[width=\linewidth]
    {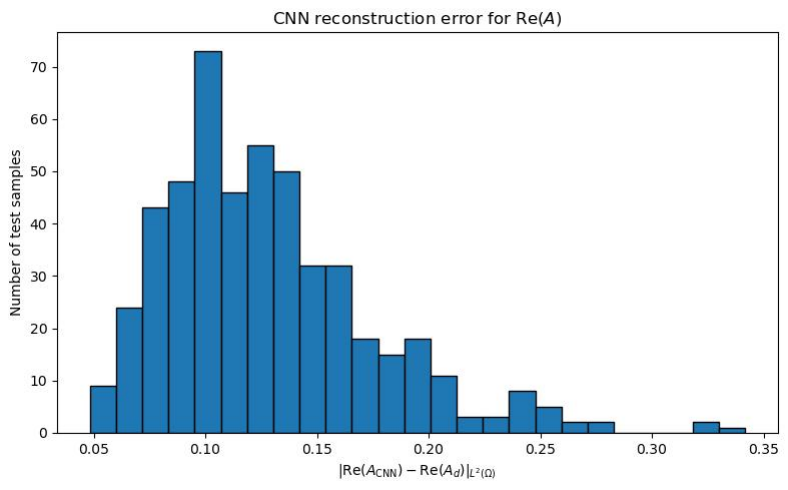}
    \caption{$L^2(\Omega)$ reconstruction errors for
    $\operatorname{Re}(A)$.}
\end{subfigure}
\hfill
\begin{subfigure}{0.32\textwidth}
    \centering
    \includegraphics[width=\linewidth]
    {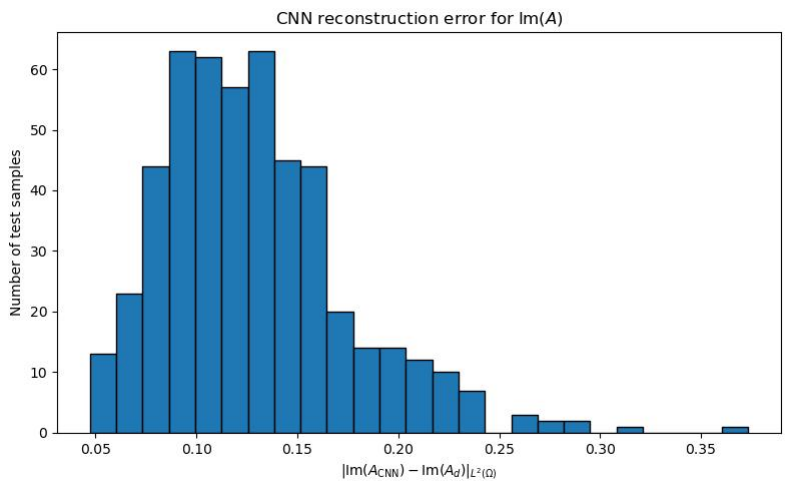}
    \caption{$L^2(\Omega)$ reconstruction errors for
    $\operatorname{Im}(A)$.}
\end{subfigure}
\hfill
\begin{subfigure}{0.32\textwidth}
    \centering
    \includegraphics[width=\linewidth]
    {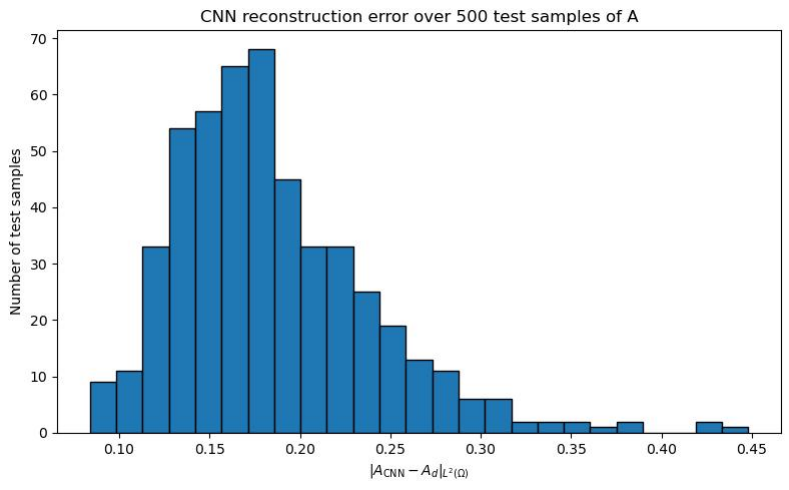}
    \caption{$L^2(\Omega)$ reconstruction errors for the
    complex-valued coefficient $A$.}
\end{subfigure}

\caption{Distribution of the CNN reconstruction errors over all $500$
combined-shape test samples using internal displacement data contaminated
with $1\%$ relative noise. The CNN is trained using noiseless individual
Gaussian and mollified-ellipse samples and is applied to the noisy
combined-shape measurements without retraining.}
\label{fig:cnn_exp2_noise1_error_histograms}

\end{figure}

To examine the performance of the complete hybrid framework in the
presence of measurement noise, we consider a representative combined-shape
coefficient from the test dataset. In this example, the Gaussian and
mollified-ellipse inclusions are spatially separated in the real component
of the coefficient, while partial overlap occurs in the imaginary component.
The corresponding displacement measurement is contaminated with $1\%$
relative noise and supplied to the pretrained CNN. The resulting CNN
prediction is then used as the initial iterate for the PDE-constrained
optimization.

The optimization and line-search parameters used for this reconstruction
are summarized in Table~\ref{tab:experiment2_optimization_parameters}.

Figure~\ref{fig:exp2_noise1_reconstruction} compares the true coefficient,
the CNN-generated initial estimate, and the final hybrid reconstruction.

\begin{figure}[H]
    \centering
    \includegraphics[width=0.7\textwidth]
    {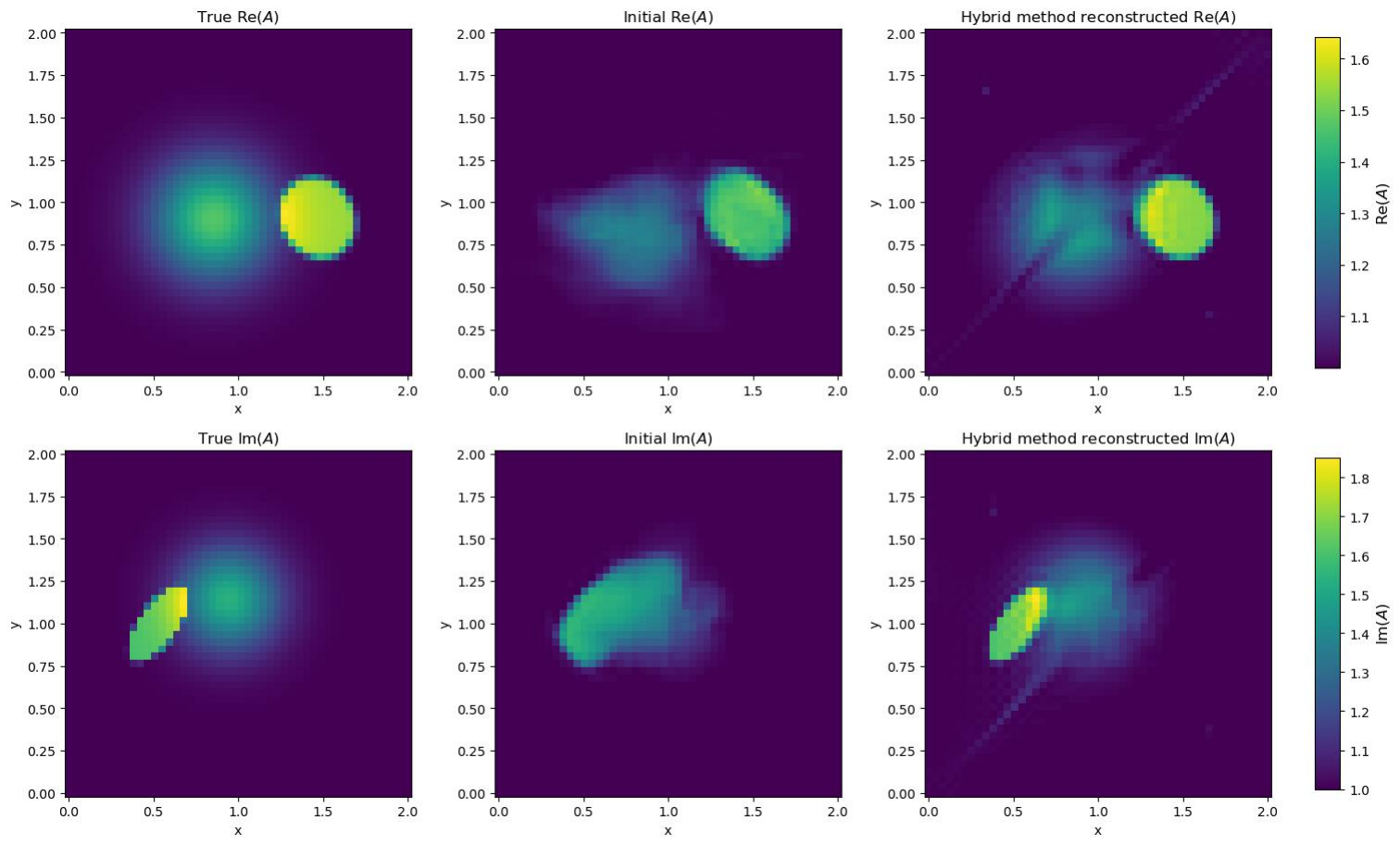}
    \caption{Reconstruction of a representative combined-shape coefficient
    from internal displacement data contaminated with $1\%$ relative noise.
    The real component contains spatially separated Gaussian and
    mollified-ellipse inclusions, whereas the inclusions partially overlap
    in the imaginary component. The left, middle, and right columns show
    the true coefficient, the CNN-generated initial estimate, and the final
    hybrid reconstruction, respectively. The top and bottom rows correspond
    to the real and imaginary components of $A(x,y)$.}
    \label{fig:exp2_noise1_reconstruction}
\end{figure}

The corresponding coefficient reconstruction errors are reported in
Table~\ref{tab:exp2_noise1_errors}.

\begin{table}[H]
\centering
\caption{Reconstruction errors for the representative combined-shape
coefficient using displacement data contaminated with $1\%$ relative noise.}
\label{tab:exp2_noise1_errors}

\begin{tabular}{lccc}
\hline
\textbf{Reconstruction error}
& \textbf{CNN initialization}
& \textbf{Hybrid reconstruction}
& \textbf{Error decrease (\%)} \\
\hline

$\|\operatorname{Re}(A)-\operatorname{Re}(A_d)\|_{L^2(\Omega)}$
& $1.229673\times10^{-1}$
& $8.524580\times10^{-2}$
& $30.68$ \\

$\|\operatorname{Im}(A)-\operatorname{Im}(A_d)\|_{L^2(\Omega)}$
& $1.193186\times10^{-1}$
& $6.553128\times10^{-2}$
& $45.08$ \\

$\|A-A_d\|_{L^2(\Omega)}$
& $1.713414\times10^{-1}$
& $1.075230\times10^{-1}$
& $37.24$ \\

\hline
\end{tabular}
\end{table}

For this combined-shape example, the PDE-constrained optimization improves
the CNN initialization despite the presence of $1\%$ measurement noise.

The improvement is also visible in the reconstructed coefficient fields.
For the real component, the hybrid reconstruction preserves the two
spatially separated inclusions and provides a sharper representation of
their geometry. For the imaginary component, the optimization improves the
distinction between the partially overlapping Gaussian and
mollified-ellipse structures. Thus, although the CNN was trained only on
noiseless samples containing individual inclusion types, its prediction
provides an effective initialization for the PDE-constrained reconstruction
of a more complex coefficient configuration in the presence of measurement
noise.

\section{Conclusion}

In this work, we developed a hybrid framework for reconstructing spatially varying complex-valued shear modulus functions from internal displacement measurements in a two-dimensional model of magnetic resonance elastography. The approach combines a convolutional neural network (CNN) with adjoint-based PDE-constrained optimization, using the CNN to provide a spatially informed initialization that is subsequently refined by a nonlinear conjugate-gradient method with Armijo line search. Numerical experiments demonstrate that the CNN provides useful initial reconstructions across different inclusion geometries and can generalize to configurations not encountered during training, while the PDE-constrained optimization improves reconstruction accuracy in the noiseless setting and enforces consistency with the governing physical model. Experiments with noisy measurements show that the CNN retains dominant spatial features under moderate perturbations, although the optimization becomes increasingly sensitive to noise and may introduce spatial artifacts, emphasizing the importance of appropriate regularization and stopping criteria. Overall, the results demonstrate the potential of combining data-driven initialization with physics-based optimization to reduce dependence on arbitrary initial guesses while retaining the interpretability of a mechanistic model. Future work will focus on noise-aware regularization, improved learning and optimization strategies, more heterogeneous coefficient geometries, and extensions to realistic three-dimensional MRE models.

\bibliographystyle{unsrt}
\bibliography{Mrebib}

\end{document}